\documentclass[reqno]{amsart}
\usepackage{amsmath, amsthm, amssymb, amstext}
\usepackage[foot]{amsaddr}
\usepackage[left=2cm,right=2cm,top=1.5cm,bottom=1.5cm,includeheadfoot]{geometry}
\usepackage{hyperref,xcolor}
\hypersetup{
	pdfborder={0 0 0},
	colorlinks,
}

\usepackage{todonotes}
\usepackage{enumitem}
\usepackage{lineno}
\newtheorem{theorem}{Theorem}

\newtheorem{lemma}[theorem]{Lemma}
\newtheorem{proposition}[theorem]{Proposition}

\newcommand{\DD} {\displaystyle}
\newcommand{\la} {\lambda}

\newcommand{\eps}{\varepsilon}

\newcommand{\Om}{\Omega}

\newcommand{\weak}{\rightharpoonup}

\newcommand{\al} {\alpha}

\newcommand{\ga} {\gamma}
\newcommand{\Ga} {\Gamma}

\newcommand{\De} {\Delta}

\newcommand{\noi} {  \noindent}
\newcommand{\na} {\nabla}

\newcommand{\mb} {\mathbb}
\newcommand{\mc} {\mathcal}

\numberwithin{theorem}{section} \numberwithin{equation}{section}

\title[ ]{On the existence results for $m$-Harmonic equation with  critical choquard Nonlinearity }

\author[Abhishek, Sarika Goyal ]
{Abhishek and Sarika Goyal }
\address{Abhishek \newline
	Department of Mathematics, Netaji Subhash University of Technology, Dwarka Sector-3, New Delhi 110078, India}
\email{abhishekhoon1998@gmail.com, abhishek.phd23@nsut.ac.in}

\address{Sarika Goyal \newline
	Department of Mathematics, Netaji Subhash University of Technology, Dwarka Sector-3, New Delhi 110078, India}
\email{sarika@nsut.ac.in, sarika1.iitd@gmail.com }

\subjclass[2020]{35J35; 35B33; 35J60; 35J91}
\keywords{Polyharmonic operator, Choquard Nonlinearity, Hardy-Littlewood-Sobolev critical exponent, Subcritical Perturbation. }
\begin{document}
\maketitle
\begin{abstract}
This article established the existence results for the $m$-harmonic equation involving critical Choquard nonlinearity and subcritical perturbation.
We first explore the minimizers of the $m$-harmonic operator with the critical Choquard equation. Then, using these minimizers, we establish delicate estimates to show the energy below the threshold level, which helps to recover the compactness. Further, we prove the existence of a nontrivial solution for our problem with different kinds of local and nonlocal subcritical perturbations. To the best of our knowledge, this is the first article dealing with the polyharmonic equation and critical Choquard type nonlinearity. The results obtained are even new for $m\geq 2$.
\end{abstract}

\section{Introduction}
\noi In this article, our main goal is to investigate the existence of a nontrivial solution to the following $m$-harmonic equation with critical Choquard nonlinearity and subcritical perturbation  
\begin{equation}\label{e1}
    \left\{
\begin{aligned}
(-\Delta)^m u &= \left(\DD\int_{\Om}\frac{|u(y)|^{2_{\al,m}^{*}}}{|x-y|^{\al}}dy\right)  {|u(x)|^{2_{\al,m}^{*}-2} u(x)}+ \lambda f(u) \quad \;&& \text{in}\; \Om,\\
 \nabla^k u&=0\; \forall \; |k|\leq m-1&& \text{on}\ \partial\Om,
\end{aligned}\right.
\end{equation}
where $\Om$ is bounded domain of $\mb{R}^N$ with a smooth boundary, $N\geq2m+1$, $m\in \mb{N}$, $0<\al<N$, $\la>0$ is a parameter, and $2_{\al, m}^{*}=\frac{2N-\al}{N-2m}$ is the critical exponent in the sense of the Hardy--Littlewood--Sobolev inequality. The $m$-harmonic operator $\Delta^{m}$ is defined as
\[\Delta^{m} u =
\begin{cases}
\Delta^{j}\left(\Delta^{j} u\right) & \text{if } m = 2j,\quad j = 1,2,\ldots, \\[6pt]
\nabla. (\Delta^{\, j-1} \nabla \Delta^{\, j-1} u)
& \text{if } m = 2j - 1,\quad j = 1, 2, \ldots
\end{cases}\]
\noi This is also called a polyharmonic operator. In the last few years, remarkable research efforts have been directed toward the study of biharmonic and polyharmonic equations, which are motivated both by their mathematical richness and their various applications in physics and chemistry, such as phase field models of multiphase systems, micro-electro-mechanical systems, thin film theory, interface dynamics, 
flow in Hele-Shaw cells, nonlinear surface diffusion on solids and the
deformation of a nonlinear elastic beam (see \cite{23ferrero2009solutions, 24myers1998thin}).
 A comprehensive review of the literature exists on biharmonic or polyharmonic problems, see \cite{31colasuonno2012multiple, 29edmunds1990critical, 35grunau1996conjecture} and the references therein.
Specifically, Pucci and Serrin\cite{36pucci1986general} considered the following  polyharmonic equation with critical nonlinearity
\begin{equation}\label{e55}
\left\{
\begin{aligned}
(-\Delta)^m u &= |u|^{2^*_{m}-2}u + \lambda u &&  \text{ in}\  \Om, \\
 {\nabla}^k u&=0 \,  \forall \ |k|\leq m-1 \,&& \text{ on} \ \partial\Om.
\end{aligned}\right.
\end{equation}

\noi They proved that problem \eqref{e55} exhibits only trivial solution $u=0$ if $\la<0$ and
$\Om$ is star-shaped. In \cite{32pucci1990critical}, authors also established the existence of a nontrivial radial solution when $\Om$ is a unit ball. Gazzola \cite{33gazzola1998critical} proved that if $N=2m+1,\ldots,4m-1$,
\eqref{e55} admits a nontrivial solution for all
$\la \in (\la_1^{m} - \la_m^2, \la_1^{m})$. In \cite{39ge2011critical}, Ge and Zhou proved the existence of a solution for a critical problem with the polyharmonic operator. Later in \cite{40dou2012solutions}, Dou, Guo and Wang  extend the results of Pucci and Serin \cite{36pucci1986general}, proving the existence of nontrivial and sign-changing solutions to equation \eqref{e55}.  While in \cite{22shang2014multiple}, Shang and Li discussed the existence of two nontrivial solutions by extracting the $(PS)$ sequence in the Nehari Manifold. Recently, Chen and Wang \cite{38chen2025new} proved a different type of solution using the reduction argument and local Pohozaev identity.\\
In \cite{1310.57262/ade/1366030750}, Lakkis explored the existence result to the following $m$-harmonic equation with critical nonlinearity in critical dimension $n=2m$: 
\begin{equation*} 
\begin{cases}
(-\Delta)^{m} u=g(u) & \text {in}\ \Om,\\
u=\nabla u=..\nabla^{m-1}=0 & \text {on}\ \partial\Om, 
\end{cases}
\end{equation*}
where $g(u)$ is a $C^1$ map that satisfies $\frac{\partial g}{\partial u}> \frac{g(u)}{u}$ for $u \in \mb{R}\setminus\{0\}$. Moreover, the growth of the nonlinearity $g(u)$ is similar to $exp(|t|^{\frac{n}{n-m}})$ at infinity. Moreover, In \cite{14article}, Lam and Lu established the existence of nontrivial solutions for polyharmonic equations with exponential nonlinearity.\\

\noi In recent years, an extensive study has been done on the elliptic equations involving Choquard-type nonlinearity due to their wide range of applications. 
Gao and Yang \cite{37article} 
explored the following Brezis-Nirenberg type critical problem for the nonlinear Choquard equation
\begin{equation}\label{e59}
    -\Delta u = \left(\DD\int_{\Om}\frac{|u(y)|^{2_{\al}^{*}}}{|x-y|^{\al}}dy\right)  {|u(x)|^{2_{\al}^{*}-2} u}+ \la u \text{ in} \, \Om, \quad u=0\;\mbox{on}\;\partial \Om,
\end{equation}
where 
 $\Om$ is an open and bounded subset of $\mb{R}^N$ with Lipschitz boundary, $N\geq3$, $\al \in (0,N)$ and $\la$ is a parameter. They showed the existence and nonexistence of the nontrivial solution to the equation \eqref{e59}. In  \cite{45article}, the authors proved the existence and asymptotic behaviour of nontrivial solutions for the Laplace equation involving the critical Choquard equation with potential. Moreover, Du and Yang in \cite{44Du2019UniquenessAN} investigated the equation \eqref{e59} for $\la>0$  and proved the existence, uniqueness and nondegeneracy of the solution. For more literature on
Choquard equations, we refer \cite{46battaglia2017existence,47gao2020existence,48gao2020existence}, etc.
Recently, Chen and Wang \cite{49chen2026multiple} extended the classical local Brezis-Nirenberg problem to a nonlocal setting via the Hardy-Littlewood-Sobolev convolution term. In contrast to the fractional Laplacian, where nonlocality arises from the operator and the nonlocal effect arises from the nonlinearity itself, which strongly influence the behavior of the Choquard equation.

\vspace{.2cm}
\noi Motivated by the work of Gao and Yang \cite{37article, 1gao2017nonlocal}, in \cite{43rani2022multiple}, Rani and Goyal  
studied the multiplicity results to the following biharmonic equation  critical Choquard nonlinearity with sign-changing weight functions
\begin{equation} \label{e58}
 \begin{cases}
\Delta^2 u = \la f(x) |u|^{q-2} +g(x)\left(\DD\int_{\Om}\frac{g(y)|u(y)|^{2_{\al}^{*}}}{|x-y|^{\al}}dy\right)  {|u(x)|^{2_{\al}^{*}-2} u(x)} \quad \;& \text{in}\; \Om,\\
u, \nabla u=0\   &\text{on}\ \partial\Om.
\end{cases}   
\end{equation}
Then, using the Nehari manifold and fibering map analysis, they proved the existence of the multiplicity result with respect to parameter $\lambda$. Later, Lei and Suo \cite{41lei2025multiple} extend the results of equation \eqref{e58}
 and showed the existence of a $k$-nontrivial solution. Moreover, very few articles are available on the Biharmonic operator with Choquard equations (see\cite{42ji2026normalized,51liu2024existence}). But to the best of our knowledge, no work has been done so far dealing with the polyharmonic operator with critical Choquard nonlinearity. 

 \vspace{.07cm}
\noi With this available literature, one natural question arise: Can one exblish the existence results for polyharmonic operator with critical Choquard nonlinearity? If yes, then what will be the minimizer for $m$-harmonic operator with this nonlinearity? We answer these questions positively in this article.  With the help of these minimizers, we showed the existence of a nontrivial solution for the polyharmonic critical Choquard equation \eqref{e1} with different kinds of subcritical perturbations on the bounded domain $\Om \subset \mb{R}^N$.

\vspace{.2cm}
\noi Now, we state the main results of the article with different perturbation:

\vspace{.07cm}
\noi For the subcritical term $f(x)=|u|^{q-2}u, 1\leq q<2^{*}_m-1$ in \eqref{e1}, problem can be rewritten as\\
\begin{equation}\label{e5}
\left \{
\begin{aligned}
(-\Delta)^m u &= \left(\DD\int_{\Om}\frac{|u(y)|^{2_{\al,m}^{*}}}{|x-y|^{\al}}dy\right)  {|u(x)|^{2_{\al,m}^{*}-2} u(x)}+ \la |u|^{q-2}u \quad \,&& \text{in}\; \Om,\\
\nabla^k u &=0 \ \forall \  |k|\leq m-1 \,&& \text{on}\ \partial\Om.
\end{aligned}\right.
\end{equation}
Then we obtain the following existence result for different exponents of $q$.
\begin{theorem}\label{t1}
If $1<q<2^{*}_m-1$, $N\geq2m+1$ and $0<\al<N$. Then, problem \eqref{e5} has at least one nontrivial solution provided that either
\begin{enumerate}
\item  $N>\max \left\{\min \left\{\frac{2m(q+3)}{q+1}, 2m+\frac{\al}{q+1}\right\}, \frac{2m(q+1)}{q}\right\}$ and $\la>0$, or
\item  $N \leq \max \left\{\min \left\{\frac{2m(q+3)}{q+1}, 2m+\frac{\al}{q+1}\right\}, \frac{2m(q+1)}{q}\right\}$ and $\la$ is sufficiently large.
\end{enumerate}
\end{theorem}
\begin{theorem} \label{t4}
If $q=1$, $N\geq2m+1$ and $0<\al<N$. Then, the following results hold:
\begin{enumerate}
\item  If $N\geq4m$, then \eqref{e5} has a nontrivial solution for $\la>0$, provided $\la$ is not an eigenvalue of $(-\Delta)^m$ with homogeneous Dirichlet boundary condition.\\
\item If $2m+1 \leq N \leq 4m-1$, then there exist $\la^*$ such that \eqref{e5} has a nontrivial solution for $\la>\la^*$, provided $\la$ is not an eigenvalue of 
$(-\Delta)^m$ with homogeneous Dirichlet boundary condition.\\ 
\end{enumerate}
\end{theorem}
\noi The existence of solutions depends on the dimension $N$ and parameter $\la>0$ with respect to the spectrum of the Dirichlet Laplacian.
If $N\geq2m+1$ and $\la<\la^{m}_1$, then the problem admits a nontrivial weak solution. The proof is variational and relies on a Mountain Pass construction, the corresponding minimax level lies below the critical Hardy–Littlewood-Sobolev constant. If $N\geq 2m+1$ and $\la>\la^{m}_1$ where $\la$ is not an eigenvalue, the functional is indefinite. In this case, the Linking argument based on the spectral decomposition of $\la$ yields a nontrivial solution. In dimension $N\in [2m+1, 4m)$, there exists $\la^*>0$ such that solution exist for all $\la>\la^*$.\\ 

\noi Next, we study the problem \eqref{e1} with a subcritical nonlocal term $f(u)= \left(\DD\int_{\Om} \frac{|u|^q}{|x-y|^\al} d y\right)|u|^{q-2} u$, that is
\begin{equation}\label{e6}
\left\{
\begin{aligned}
(-\Delta)^m u&=\left(\DD\int_{\Om} \frac{|u|^{2_{\al, m}^*} }{|x-y|^\al} d y\right)|u|^{2_{\al, m}^*-2} u+\la\left(\DD\int_{\Om} \frac{|u|^q}{|x-y|^\al} d y\right)|u|^{q-2} u \,&& \text{ in } \Om, \\
\nabla^k u&=0\ \forall \ |k|\leq m-1 \,&&  \text { on}\   \partial\Om. 
\end{aligned}\right.
\end{equation}
For this case, we establish the following existence result:
\begin{theorem}\label{t2}
If $1<q<2_{\al, m}^*$, $N\geq2m+1$ and $0<\al<N$. Then, problem \eqref{e6} has at least one nontrivial solution provided that either \\
\item (1) $N>\frac{2m(q+1)-\al}{q-1}$ and $\la>0$,
\item (2) $N \leq \frac{2m(q+1)-\al}{q-1}$ and $\la$ is sufficiently large.
\end{theorem}
\noi While giving the proofs of the Theorems \ref{t2}, the main challenge originates from the presence of two nonlocal nonlinearities. One of which dominates the Hardy--Littlewood--Sobolev critical growth. Such  double nonlocal structure not only demolishes the compactness at the critical level but also, it complicates the estimation of the corresponding minimax energy. To address these difficulties, we first establish the compactness of the $(PS)$ sequence below an appropriate critical threshold, and then we obtain the estimates for the functional using suitable extremal functions. These estimates ensure that the mountain pass level lies below the critical level, which helps to obtain the existence results.\\
\noi Similarly, we study the problem with the Sobolev critical exponent and subcritical nonlocal perturbation.
\begin{theorem}\label{t3}
If $1<q<{2_{\al, m}^*}-1$, $N\geq2m+1$ and $0<\al<N$. Then, the problem 
\begin{equation*}
\left\{
\begin{aligned}
(-\Delta)^m u&=\la\left(\DD\int_{\Om} \frac{|u|^q}{|x-y|^\al} d y\right)|u|^{q-2} u+|u|^{2_{\al, m}^*-2} u \,&& \text{ in}\ \Om, \\
\nabla^k u&=0\  \forall \ |k|\leq m-1 \,&& \text { on}\ \partial\Om, 
\end{aligned}\right.
\end{equation*}
has at least one nontrivial solution, provided that either
\item (1) $N>\frac{2m(q+1)-\al}{q-1}$ and $\la>0$,
\item (2) $N \leq \frac{2m(q+1)-\al}{q-1}$ and $\la$ is sufficiently large.
\end{theorem}

\noi We use the following notation and definition throughout the paper:
\begin{itemize}
    \item Throughout the paper $ C^{\prime}, C, C_1, C_2, C_3, C_4, C_5,....$ denotes positive constants.\\
    \item We denote the standard norm $|\cdot|_p$ for the $L^p$ norm for $p\in [0, \infty]$.\\
    \item  We denote strong convergence and weak convergence by $\to$\ and $\weak$ respectively. 
    \item Let $X$ be a real Hilbert space and $\mc{I}\in C^1(X, \mb{R})$. A sequence $(u_n)\subset X$ is a Palais-Smale sequence at level $c$ (or $(PS)_c$ sequence) if, \[\mc{I}(u_n) \to c \quad \text{ and} \quad \mc{I}^{\prime}(u_n) \to 0, \text{ as} \quad n\to \infty.\] 
\noi The functional $\mc{I}$ is said to satisfy the Palais–Smale condition at level $c$ if every $(PS)_c$ sequence has a convergent subsequence in X.
\end{itemize}
This paper is organised as follows: 
In section 2, we give some preliminary results and prove the important estimates with the help of the minimizers. 
 In the last three sections-$3$, $4$ and $5$, we obtain the existence results for different types of perturbation.
\section{\bf{Minimizers Section}}
\noi In this section, we introduce the functional setting and some basic results. Moreover, we determine the key estimates with the help of the minimizers. Firstly, we define the functional space,
\begin{align*}
\mb{H}^m_0(\Om):= \text{ Closure of } {C_c^\infty(\Omega)} \text{ functions with respect to the norm } \|\cdot\|_{\mb{H}^m(\Om)},\end{align*}  is a Hilbert space equipped with the norm $\|u\|^2 = \int_{\Om} |D^m u|^2 dx, $
where\[D^m u = \begin{cases}
(-\Delta)^\frac{m}{2}u , & \text{if } m = 2j,\quad j = 1, 2, \ldots ,\\
\nabla(-\Delta)^\frac{(m-1)}{2} u , & \text{if } m = 2j - 1,\quad j = 1, 2, \ldots,
\end{cases}\]
and $\mb{H}^m(\Om)= \left\{u\in L^2(\Om): D^{\al} u\in L^2(\Om), \text {where} \,  0\leq|\al|\leq m \right\}$  with the norm $\|u\|_{\mb{H}^m(\Om)} = ( \|u\|_{L^2}^2 + \|D^m u\|_{L^2}^2 ),$ is a Hilbert space.\\
\begin{proposition}(Hardy-Littlewood-Sobolev inequality)
Let $t, s>1$ and $0<\al<N$ with $\frac{1}{s}+\frac{\al}{N}+\frac{1}{t}= 2$, for $g\in L^t(\mb{R}^N)$ and $h\in L^s(\mb{R}^N)$. Then there exist a sharp constant $C(t, N, \al, s),$ independent of $g$ and $h$ such that
\begin{equation}\label{e2}
\int_{\mb{R}^N} \int_{\mb{R}^N} \frac{g(x) h(y)}{|x-y|^\al} \, dx \, dy \leq C(t, N, \al, s) \|g\|_{L^t(\mb{R}^N)} \|h\|_{L^s(\mb{R}^N)}.
\end{equation}
\end{proposition}
\noi If $t=s=\frac{2N}{2N-\al}$, then
  \[C(t, N, \al, s)=C(N, \al)=\pi^{\frac{\al}{2}}\frac{\Gamma(\frac{N}{2} -{\frac{\al}{2}})}{\Ga(N-{\frac{\al}{2}})}\left(\frac{\Ga(\frac{N}{2})}{\Ga{(N)}}\right)^{\frac{\al}{N}-1},\] 
 such that equality holds in \eqref{e2} if and only if $g = C h$ and
\[h(x) = A \left( b^2 + |x - a|^2 \right)^{\frac{2N-\al}{2}}, \] for some $A \in \mb C, 0\not=b\in\mb{R}$ and $a \in\mb{R}^N.$
Suppose $g=h=|u|^q$, then by Hardy-Littlewood-Sobolev inequality,
$\int_{\mb{R}^N} \int_{\mb{R}^N} \frac{|u(x)|^{q}|u(y)|^{q}}{|x - y|^{\al}} \, dx \, dy$ is well defined for $|u|^{q} \in L^s(\mb{R}^N)$ with $s>1$ and $\frac{2}{s} + \frac{\al}{N}=2$. Thus, for $u \in\mb {H}_0^m(\mb{R}^N)$, $s$ satisfies the following inequality:
\[\frac{2N - \al}{N} \leq q \leq \frac{2N- \al}{N- 2m} := 2_{\al, m}^*, \]
where $\frac{2N- \al}{N}$ and $2_{\al, m}^*$ are known as the lower and upper critical exponents, respectively, in the sense of the Hardy-Littlewood-Sobolev inequality.\\
\noi Now, for all $u \in\mb{H}_0^m(\mb{R}^N)$, by the Hardy-Littlewood-Sobolev inequality, it is observed that
\begin{equation}\label{e3}
\left({\int_{\mb{R}^N} \int_{\mb{R}^N} \frac{|u(x)|^{2_{\al, m}^*}|u(y)|^{2_{\al, m}^*}}{|x - y|^\al} \, dx\, dy } \right)^ \frac{1}{2_{\al, m}^*}
\leq (C(N, \al))^{\frac{1}{2_{\al, m}^*}} \|u\|_{2_m^*}^{2}.
\end{equation}
\begin{lemma} \label{l17}
Assume $N\geq2m+1$ and $0<\al<N$. Let
\[\|\cdot\|_{NL}:= \left(\int_{\Om} \int_{\Om} \frac{|\cdot|^{2^*_{\al,m}}|\cdot|^{2^*_{\al,m}}}{|x-y|^\al} \,dx \,dy\right)^{\frac{1}{22^*_{\al, m}}}\] 
and \[X_{NL}:=\{u:\Om \to \mb{R} \text{ is measuraable}, \|u\|_{NL} <+\infty\}.\]
Then $\|.\|_{NL}$ is a norm in $X_{NL}$. Moreover, under the norm $\|.\|_{NL}$, $X_{NL}$ is a Normed space.  
\end{lemma}
\begin{proof}
By the semigroup property of the Riesz potential (see \cite{28riesz1949integrale}), we obtain
\[\int_{\Om} \int_{\Om} \frac{|u(x)|^{2^*_{\al,m}}|u(y)|^{2^*_{\al,m}}}{|x-y|^\al} d x d y=\int_{\Om}\left(\int_{\Om} \frac{|u(y)|^{2^*_{\al,m}}}{|x-y|^{\frac{N +\al}{2}}} d y\right)^2 d x,\]
for every $u \in \mb{H}_0^m(\Om)$. Then, Minkowski's inequality, yields for any $x \in \Om$,
\begin{align*}
\left(\int_{\Om} \frac{|u_1(y)+u_2(y)|^{2^*_{\al,m}}}{|x-y|^{\frac{N+\al}{2}}} d y\right)^2 & =\left(\int_{\Om}\left|\frac{u_1(y)}{|x-y|^{\frac{N+\al}{22^*_{\al,m}}}}+\frac{u_2(y)}{|x-y|^{\frac{N+\al}{22^*_{\al,m}}}}\right|^{2^*_{\al,m}} d y\right)^{\frac{1}{2^*_{\al,m}}22^*_{\al,m}} \\
& \leq\left(\left(\int_{\Om} \frac{|u_1(y)|^{2^*_{\al,m}}}{|x-y|^{\frac{N+\al}{2}}} d y\right)^{ \frac{2}{22^*_{\al,m}}}+\left(\int_{\Om} \frac{|u_2(y)|^{2^*_{\al,m}}}{|x-y|^{\frac{N+\al}{2}}} d y\right)^{\frac{2}{22^*_{\al,m}}}\right)^{2 2^*_{\al,m}}.
\end{align*}
Notice that the integrals are non-negative and so, by Minkowski's inequality again, we have
\begin{align*}
\left(\int_{\Om}\left(\int_{\Om} \frac{|u_1(y)+u_2(y)|^{2^*_{\al,m}}}{|x-y|^{\frac{N+\al}{2}}} d y\right)^2 d x\right)^{\frac{1}{22^*_{\al,m}}}\leq&\left(\int_{\Om}\left(\int_{\Om} \frac{|u_1(y)|^{2^*_{\al,m}}}{|x-y|^{\frac{N+\al}{2}}} d y\right)^2 dx\right)^{\frac{1}{22^*_{\al,m}}}\\ & \quad +\left(\int_{\Om}\left(\int_{\Om} \frac{|u_2(y)|^{2^*_{\al,m}}}{|x-y|^{\frac{N+\al}{2}}} d y\right)^2 dx\right)^{\frac{1}{22^*_{\al,m}}},
\end{align*}\\
i.e.,
\[\|u_1 +u_2\|_{NL} \leq \|u_1\|_{NL} +\|u_2\|_{NL},\]\\
for every $u_1,u_2 \in L^{2^*_m}(\Om)$ and the remaining property, nonnegativity and absolute homogeneity follow trivially. So, $\|\cdot\|_{NL}$ is a norm. 
\end{proof} 
\noi The following Brezis--Lieb type lemma for the nonlocal term is proved in \cite{2ackermann2004periodic}(the subcritical case) and \cite{1gao2017nonlocal}(the critical case).
\begin{lemma}\label{l3}
Let $N\geq2m+1$, $0<\al<N$. If $u_n$ is a bounded sequence $\{u_n\}$ in $L^{2_m^*}(\Om)$ such that $u_n\to u$ a.e. in $\Om$ as $n \to\infty,$ then the following holds, as $n \to\infty$
\[\int_{\Om}(|x|^{-\al}*|u_n|^{2_{\al, m}^*})|u_n|^{2_{\al, m}^*} dx- \int_\Om(|x|^{-\al}*|u_n -u|^{2_{\al, m}^*})|u_n-u|^{2_{\al, m}^*} dx \to \int_\Om(|x|^{-\al}*|u|^{{2_{\al, m}^*}})|u|^{2_{\al, m}^*} dx.\]
\end{lemma}
\begin{proof}
Firstly, the proof proceeds in a  similar manner to the proof of the Brezis--Lieb lemma (see \cite{37article}), we know that
\[
|u_n - u|^{2_{\al, m}^*} - |u_n|^{2_{\al, m}^*} \to |u|^{2_{\al, m}^*}
\text{ in}  \, L^{\frac{2N}{2N-\al}}(\mb{R}^N) \text{ as} \,  n \to \infty.\]
 The Hardy--Littlewood--Sobolev inequality implies that
\[
|x|^{-\al} * \left(|u_n - u|^{2_{\al, m}^*} - |u_n|^{2_{\al, m}^*}\right)
\to
|x|^{-\al} * |u|^{2_{\al, m}^*}
\text{ in}  \, L^{\frac{2N}{\al}}(\mb{R}^N) \text{ as} \,  n \to \infty.\]
On the other hand, we notice that 
\begin{align*}
\int_{\mb{R}^N}
\left(|x|^{-\al} * |u_n|^{2_{\al, m}^*}\right)|u_n|^{2_{\al, m}^*}\,dx
-
&\int_{\mb{R}^N}
\left(|x|^{-\al} * |u_n-u|^{2_{\al, m}^*}\right)|u_n-u|^{2_{\al, m}^*}\,dx \\
&=
\int_{\mb{R}^N}
\left(|x|^{-\al} *
\left(|u_n|^{2_{\al, m}^*}-|u_n-u|^{2_{\al, m}^*}\right)\right)
\left(|u_n|^{2_{\al, m}^*}-|u_n-u|^{2_{\al, m}^*}\right)\,dx \\
&\quad
+2\int_{\mb{R}^N}
\left(|x|^{-\al} *
\left(|u_n|^{2_{\al, m}^*}-|u_n-u|^{2_{\al, m}^*}\right)\right)
|u_n-u|^{2_{\al, m}^*}\,dx \\
\end{align*}
Since, we have
$|u_n-u|^{2_{\al, m}^*} \weak 0$
in $L^{\frac{2N}{2N-\al}}(\mb{R}^N)$ as $n \to \infty$.
Therefore, as $n\to\infty$
\[\int_{\Om}(|x|^{-\al}*|u_n|^{2_{\al, m}^*})|u_n|^{2_{\al, m}^*} dx- \int_\Om(|x|^{-\al}*|u_n -u|^{2_{\al, m}^*})|u_n-u|^{2_{\al, m}^*} dx \to \int_\Om(|x|^{-\al}*|u|^{{2_{\al, m}^*}})|u|^{2_{\al, m}^*} dx.\]
This completes the proof.
\end{proof}

\vspace{0.02cm}
\noi Let $S$ be the best constant defined as
\[S := \inf_{u \in\mb{H}_0^m(\Om) \setminus \{0\}} \frac{\int_\Om|D^m u|^2 dx}{\left( \int_\Om |u(x)|^{2_m^*} dx \right)^\frac{2}{2_m^*}},\]
where $2_m^* = \frac{2N}{N-2m}$. Then it is well known that $S$ is achieved if and only if $\Om = \mb{R}^N$, by the function \cite{26swanson1992best}
\[U(x) = \frac{C_{N,m}^\frac{N-2m}{4m}}{\left(1+|x|^2\right)^{\frac{N-2m}{2}}}.\]
All minimizer of $S$ are obtained by
\begin{equation} \label{e57}
U_{\eps}(x) = \eps^ \frac{2m-N}{2}U\left( \frac{x}{\eps} \right) = \frac{C_{N, m}^{\frac{N-2m}{4m}} \, \eps^{\frac{N - 2m}{2}}}{\left( \eps^2 + |x|^2 \right)^{\frac{N - 2m}{2}}},  \text{ for } \eps > 0, \end{equation} 
where $C_{N, m}:=C(N, m) = \prod_{j=1}^m (N-2j)$.\\

\noi Next, we define $S_{H, L}$ as
\begin{align}\label{e4}
S_{H, L}:= \inf_{u \in\mb{H}_0^m(\mb{R}^N) \setminus \{0\}} \frac{\int_{\mb{R}^N} |D^m u|^2 dx}{\left( \int_{\mb{R}^N} \int_{\mb{R}^N} \frac{|u(x)|^{2_{\al, m}^*} |u(y)|^{2_{\al, m}^*}}{|x - y|^\al} \, dx \, dy \right)^\frac{1}{2_{\al, m}^*}}.
\end{align}


\begin{proposition}
The constant $S_{H, L}$ is achieved if and only if
\[u = C\left( \frac{k}{k^2 + |x - a|^2} \right)^{\frac{N - 2m}{2}},\]
where $C > 0$ is a constant, $a \in \mb{R}^N$ and $k \in \mb{R}^+$. Moreover,
\[S_{H, L} = \frac{S}{(C(N, \al))^{\frac{1}{2_{\al, m}^*}}}.\notag\]
\end{proposition}
\begin{proof}
Using the Hardy-Littlewood-Sobolev inequality, we obtain
\[S_{H, L} \geq \frac{1}{(C(N, \al))^{\frac{1}{2_{\al, m}^*}}} \inf_{u \in\mb {H}_0^m(\mb{R}^N) \setminus \{0\}} \frac{\int_{\mb{R}^N} |D^m u|^2 dx}{\left( \int_{\mb{R}^N} |u|^{2_m^*} \right)^\frac{2}{2_m^*}} = \frac{S}{(C(N, \al))^{\frac{1}{2_{\al, m}^*}}}.\]
Now with the help of the definition of $S_{H, L}$ and equation \eqref{e3}, we obtain \[S_{H, L} \leq \frac{\int_{\mb{R}^N} |D^m u|^2 dx}{\left( \int_{\mb{R}^N} \int_{\mb{R}^N} \frac{|u(x)|^{2_{\al, m}^*} |u(y)|^{2_{\al, m}^*}}{|x - y|^\al} \, dx \, dy \right)^\frac{1}{2_{\al, m}^*}}\leq\frac{\int_{\mb{R}^N} |D^m u|^2 dx}{{(C(N, \al))^{\frac{1}{2_{\al, m}^*}}}\left( \int_{\mb{R}^N} |u|^{2_m^*} \right)^\frac{2}{2_m^*}} = \frac{S}{(C(N, \al))^{\frac{1}{2_{\al, m}^*}}}.\]
Thus $S_{H, L}$ is achieved if and only if $u = C\left( \frac{k}{k^2 + |x - a|^2} \right)^{\frac{N - 2m}{2}}$, and hence 
$S_{H, L}=S(C(N, \al))^{\frac{-1}{2_{\al, m}^*}}.$
\end{proof} 
\noi Let $U(x):= \frac{C_{N,m}^{\frac{N-2m}{4m}}}{\left(1+|x|^2\right)^{\frac{N-2m}{2}}}$ be minimizers for $S$, then $\tilde{U}(x)$ be defined as \[\tilde{U}(x)=S^{\frac{(N-\al)(2m-N)}{4m(N+2m-\al)}}(C(N, \al))^{\frac{2m-N}{2(N+2m-\al)}}U_{\eps}(x), \quad \eps>0,\] is the unique minimizer for $S_{H, L}$ that satisfies the equation:
\begin{equation*}
(-\Delta)^m u = \left(\DD\int_{\mb{R}^N}\frac{|u(y)|^{2_{\al, m}^*}}{|x-y|^{\al}}dy\right) {|u|^{2_{\al,m}^*-2} u} \text{ in } \mb{R}^N,
\end{equation*}
with
\[\int_{\mb{R}^N}|D^m\Tilde{U}_\eps|^2 dx=\int_{\mb{R}^N}\int_{\mb{R}^N} \frac{|\Tilde{U}_\eps(x)|^{2_{\al, m}^*}|\Tilde{U}_\eps(y)|^{2_{\al, m}^*}}{|x-y|^\al}\, dx \, dy =(S_{H, L})^{\frac{2N-\al}{N+2m-\al}}.\]\\
\begin{lemma}
 Let $N\geq2m+1$ and $\al>0$. For any open subset $\Om$ of $\mb{R}^N$,
\[S_{H, L}(\Om):= \inf_{u \in\mb{H}_0^{m}(\mb{R}^N) \setminus \{0\}} \frac{\int_{\mb{R}^N} |D^m u|^2 dx}{\left( \int_{\mb{R}^N} \int_{\mb{R}^N} \frac{|u(x)|^{2_{\al, m}^*} |u(y)|^{2_{\al, m}^*}}{|x - y|^\al} \, dx \, dy \right)^\frac{1}{2_{\al, m}^*}} =S_{H, L},\]
where $S_{H, L}$ is never achieved except when $\Om =\mb{R}^N$.
\end{lemma}
\begin{proof}
 It is obvious that $S_{H, L}\leq S_{H, L}(\Om)$ by $\mb{H}_0^{m}(\Om)\subset \mb{H}_0^{m}(\mb{R}^N)$. Let $\{u_n\} \subset C_0^{\infty} (\mb{R}^N)$ be a minimizing sequence for $S_{H, L}$. We make translations and dilations for $\{u_n\}$ by choosing $w_n \in\mb{R}^N $ and $\tau_n >0$ such that 
\[u_n^{w_n, \tau_n}(x) :=\tau_n^{\frac{N-2m}{2}} u_n(\tau_n x+ w_n) \in C_0^{\infty} (\Om),\]
which satisfies
\[\DD\int_{\mb{R}^N}\left|D^m u_n^{w_n, \tau_n}\right|^2 dx =\DD\int_{\mb{R}^N}\left|D^m u_n\right|^2 dx\]
and
\[\left( \int_{\mb{R}^N} \int_{\mb{R}^N} \frac{|u_n^{w_n, \tau_n}(x)|^{2_{\al, m}^*} |u_n^{w_n, \tau_n}(y)|^{2_{\al, m}^*}}{|x - y|^\al} \, dx \, dy \right) =\left( \int_{\mb{R}^N} \int_{\mb{R}^N} \frac{|u(x)|^{2_{\al, m}^*} |u(y)|^{2_{\al, m}^*}}{|x - y|^\al} \, dx \, dy \right).\]
Hence we obtain $S_{H, L}(\Om)\leq S_{H, L}.$ Since $\Tilde{U}(x)$ is the only class of functions such that the equality holds in the Hardy-Littlewood-Sobolev inequality, we know that $S_{H, L}(\Om)$ is achieved only when $\Om=\mb{R}^N$.
\end{proof}
\noi Let us fix $r>0$ be a small enough such that $B_r \subset \Om \subset B_{l_0 r}$ for some positive $l_0$ and let $\phi \in C_0^{\infty}(\Om)$ be a cut-off function defined as:
\[\phi(x)= \begin{cases}1 & \text { if } x \in B_r, \\
0 & \text { if } x \in \mb{R}^N \backslash \Om \end{cases} , 
\quad 0 \leq \phi(x) \leq 1 \quad \forall x \in \mb{R}^N, \quad |D^k\phi(x)|\leq b_k\quad \forall x \in \mb{R}^N,\]
and
\begin{equation}\label{e10}
u_\eps(x) :=\phi(x) U_\eps(x), \,  \text{ for} \, \eps>0,
\end{equation} 
where $U_\eps(x)$ defined in \eqref{e57}.

\noi Then, we have the following estimates; their proof can be found in  \cite[Lemma 4.3]{22shang2014multiple} and \cite[Lemma 2] {25grunau1995positive}
\begin{equation} \label{e11}
\DD\int_{\Om}\left|D^m u_{\eps}\right|^2 dx = S^{\frac{N}{2m}}+O\left(\eps^{N-2m}\right)
\end{equation}
\noi and
\begin{equation} \label{e54}
\DD\int_\Om|u_\eps|^2dx \geq \begin{cases} C_1 \eps^{N-2m} + O(\eps^{2m})& \text { if } N<4m,\\
 C_2 \eps^{2m}|\ln \eps|+O(\eps^{2m}) & \text { if } N=4m, \\
C_3 \eps^{2m} + O(\eps^{N-2m})& \text { if } N\geq4m+1. \end{cases}
\end{equation}\


\begin{lemma}
 Let $u_\eps$ be the minimizers defined in \eqref{e10}. Then we have the following estimates:
\begin{enumerate}
\item
$\DD\int_{\Om} \DD\int_{\Om} \frac{\left|u_{\eps}(x)\right|^{2_{\al, m}^*}\left|u_{\eps}(y)\right|^{2_{\al, m}^*}}{|x-y|^\al} \, dx \, dy \geq C(N, \al)^{\frac{N}{2m}} S_{H, L}^{\frac{2N-\al}{2m}}-O\left(\eps^{\frac{2N-\al}{2}}\right)$.\\

\item $\|u_\eps \|_1= O \left(\eps^{\frac{N-2m}{2}}\right).$\\

\item $\DD\int_{\Om} u_{\eps}^{q+1} dx \geq O\left(\eps^{N-\frac{(N-2m)(q+1)}{2}}\right).$
   
\end{enumerate}
\end{lemma}
\begin{proof}
\begin{enumerate}
\item  In this case, we will estimate the convolution part; we know
 \begin{align}\label{e45}
\int_{\Om} \int_{\Om} \frac{\left|u_{\eps}(x)\right|^{2_{\al, m}^*}\left|u_{\eps}(y)\right|^{2_{\al, m}^*}}{|x-y|^\al} \, dx \, dy \geq & \int_{B_r} \int_{B_r} \frac{\left|u_{\eps}(x)\right|^{2_{\al, m}^*}\left|u_{\eps}(y)\right|^{2_{\al, m}^*}}{|x-y|^\al} \, d x \, d y \notag\\
&= \int_{B_r} \int_{B_r} \frac{\left|U_{\eps}(x)\right|^{2_{\al, m}^*}\left|U_{\eps}(y)\right|^{2_{\al, m}^*}}{|x-y|^\al} \, dx \, dy \notag\\
& =\int_{\mb{R}^N} \int_{\mb{R}^N} \frac{\left|U_{\eps}(x)\right|^{2_{\al, m}^*}\left|U_{\eps}(y)\right|^{2_{\al, m}^*}}{|x-y|^\al}-2 \int_{\mb{R}^N \backslash B_r}\int_{B_r} \frac{\left|U_{\eps}(x)\right|^{2_{\al, m}^*}\left|U_{\eps}(y)\right|^{2_{\al, m}^*}}{|x-y|^\al} \notag\\
& \quad -\int_{\mb{R}^N \backslash B_r}\int_{\mb{R}^N \backslash B_r} \frac{\left|U_{\eps}(x)\right|^{2_{\al, m}^*}\left|U_{\eps}(y)\right|^{2_{\al, m}^*}}{|x-y|^\al} \, dx \, dy \\
:= & \mathrm{A_1}-2 \mathrm{A_2}-\mathrm{A_3}.\notag
\end{align}
\noi We are going to estimate $\mathrm{A_1}, \mathrm{A_2}$ and $\mathrm{A_3}$. By direct computation and using the Hardy--Littlewood--Sobolev inequality, we know, for $\eps<1$,
\begin{equation*}
\mathrm{A_1}=\int_{\mb{R}^N} \int_{\mb{R}^N} \frac{\left|U_{\eps}(x)\right|^{2_{\al, m}^*}\left|U_{\eps}(y)\right|^{2_{\al, m}^*}}{|x-y|^\al}\, dx \, dy = C(N, \al)^{\frac{N}{2m}} S_{H, L}^{\frac{2N-\al}{2m}},
\end{equation*}
\begin{align*}
\mathrm{A_2} &=C_{N, m}^{\frac{(2N-\al)}{2m}}\eps^{2N-\al}
\int_{\mb{R}^N \setminus B_r} \int_{B_r}
\frac{1}{\left(\eps^2+|x|^2\right)^{\frac{2N-\al}{2}}
|x-y|^\al
\left(\eps^2+|y|^2\right)^{\frac{2N-\al}{2}}}\, dx \, dy \notag\\
& \leq O\left(\eps^{2N-\al}\right)\left(\int_{\mb{R}^N \setminus B_r} \frac{1}{\left(\eps^2+|x|^2\right)^{N}} d x\right)^{\frac{2N-\al}{2N}}\left(\int_{B_r} \frac{1}{\left(\eps^2+|y|^2\right)^{N}} d y\right)^{\frac{2N-\al}{2N}}\notag \\
& \leq O\left(\eps^{\frac{2N-\al}{2}}\right).
\end{align*}
Similarly, we obtain
\begin{equation}\label{e48}
\mathrm{A_3} \leq  O\left(\eps^{2N-\al}\right)
\int_{\mb{R}^N \setminus B_r} \int_{\mb{R}^N \setminus B_r}\frac{1}{\left(\eps^2 + |x|^2\right)^{\frac{(2N-\al)}{2}}|x-y|^\al\left(\eps^2 + |y|^2\right)^{\frac{2N-\al}{2}}}\, dx \, dy \quad \leq O\left(\eps^{2N-\al}\right).
\end{equation}
It follows from \eqref{e45}- \eqref{e48} that
\begin{equation}\label{e52}
\int_{\Om} \int_{\Om} \frac{\left|u_{\eps}(x)\right|^{2_{\al, m}^*}\left|u_{\eps}(y)\right|^{2_{\al, m}^*}}{|x-y|^\al}\, dx \, dy \geq C(N, \al)^{\frac{N}{2m}} S_{H, L}^{\frac{2N-\al}{2m}}-O\left(\eps^{\frac{2N-\al}{2}}\right).
\end{equation}\\ 
\item
We consider,
\[\int_\Om|u_\eps|dx=\int_{B_r}|u_\eps|dx+\int_{\Om\setminus B_r}|u_\eps|dx=I_1+I_2,\]
where
\[I_2=\int_{\Om\setminus B_r} \frac{C^\frac{N-2m}{4m}_{N, m}{\eps^\frac{N-2m}{2}}}{(\eps^2+|x|^2)^\frac{N-2m}{2}}dx= O\left({\eps^\frac{N-2m}{2}}\right)\int_{\Om\setminus B_r}\frac{1}{|x|^{N-2m}}dx=O\left({\eps^\frac{N-2m}{2}}\right)\]
and
\noi \[I_1=\int_{B_r}|u_\eps|dx\\
=\int_{B(0,\eps)}|U_\eps|dx+\int_{\eps<|x|<r}|U_\eps|dx= O\left({\eps^\frac{N-2m}{2}}\right).\]
Hence,
\begin{equation} \label{e53}
\int_\Om|u_\eps|dx= O\left({\eps^\frac{N-2m}{2}}\right).    
\end{equation}\\
\item
\begin{align*}
\int_\Om|u_\eps|^{q+1}dx&\geq\int_{B_r}|U_\eps|^{q+1}dx\\
&= O\left(\eps^{\frac{(N-2m)(q+1)}{2}}\right) \int_0^r \frac{r^{N-1}}{(\eps^2+|r|^2)^{\frac{(N-2m)(q+1)}{2}}} \, dr\\
&=O\left(\eps^{N-{\frac{(N-2m)(q+1)}{2}}}\right).
\end{align*} 
Hence, \[\int_\Om|u_\eps|^{q+1}dx\geq O\left(\eps^{N-{\frac{(N-2m)(q+1)}{2}}}\right).\]
\end{enumerate}
\end{proof}



\section{\bf{Perturbation with a superlinear local term}}
\noi In this section, we study the  following critical Choquard equation, \\
\begin{equation*}
 \left\{
 \begin{aligned}
(-\Delta)^m u &= \left(\DD\int_{\Om}\frac{|u|^{2_{\al,m}^{*}}}{|x-y|^{\al}}dy\right)  {|u|^{2_{\al,m}^{*}-2} u}+ \la u^q&& \; \text{in}\; \Om,\\
 \nabla^k u&=0\; \forall \; |k|\leq m-1 &&\; \text{on} \, \partial\Om,
\end{aligned}\right.
\end{equation*}
and established the existence of nontrivial solutions under suitable conditions, by employing the Mountain Pass geometry and compactness below a critical level.\\
For this, we first define the energy functional $\mc I_{\la}:\mb{H}_0^{m}(\Om)\to \mb{R}$ corresponding to the problem \eqref{e5} as\\
\begin{align*} 
\mc I_\la(u)= \frac{1}{2}\int_{\Om}|D^{m}u|^2 dx -\frac{\la}{q+1}\int_{\Om}|u|^{q+1} dx-\frac{1}{22^*_{\al, m}}\int_{\Om}\int_{\Om} \frac{|u(x)|^{2_{\al, m}^{*}}|u(y)|^{2_{\al, m}^{*}}}{|x-y|^{\al}} \, dx \, dy.
\end{align*}
Then $\mc{I}_{\la}$ is well defined in $C^1(\mb{H}_0^{m}(\Om))$. Moreover, the weak solution of  the problem \eqref{e5} is the critical point of $\mc{I}_{\la}$, given by
\begin{align*}
    \langle \mc{I}_\la^{\prime}(u), \phi \rangle=\int_{\Om}D^{m}u D^{m} \phi dx -\la\int_{\Om}|u|^{q-2}u{\phi} dx-\int_{\Om}\int_{\Om} \frac{|u(x)|^{2_{\al, m}^{*}}|u(y)|^{2_{\al, m}^{*}-2}u(y){\phi(y)}}{|x-y|^{\al}} \, dx \, dy=0,\quad \forall \phi \in C_c^{\infty}(\Om),
\end{align*}
where, \[\int_{\Om}D^{m}u D^{m} \phi dx =\begin{cases}\DD\int_{\Om}(-\De)^{\frac m2}u  (-\De)^{\frac m2}\phi \,dx  & \text { if } m \text { is even },\\ \\
\DD\int_{\Om}\na(-\De)^{\frac {m-1}{2}}u \cdot\na (-\De)^{\frac{ m-1}{2}}\phi \,dx &\text { if } m \text { is odd }. \\
\end{cases}\]
\noi Now, we show that $\mc{I}_{\la}$ satisfies the Mountain Pass Geometry, i.e.,
\begin{lemma}\label{l1}
If $1<q<2^{*}_{m}-1$ and ${\la}>0$, then
\item[(1)] There exist $\mu, \rho>0$ such that $\mc I_\la(u)\geq \mu$ for $\|u\|=\rho$.
\item[(2)] There exist $e\in \mb{H}^m_0(\Om)$ with $\|e\|>{\rho}$ such that $\mc I_\la(e)<0$.
\end{lemma}
\begin{proof} For $u\in\mb{H}_0^{m}(\Om)\setminus\{0\}$, we have 
$\left({ \int_\Om \int_\Om \frac{|u(x)|^{2_{\al, m}^*}|u(y)|^{2_{\al, m}^*}}{|x - y|^\al} \,dx \,dy } \right)^ \frac{1}{2_{\al, m}^*}
\leq S_{H, L}^{-1}\|u\|^2$ from \eqref{e4} and using the Sobolev embedding theorem for all, we obtain
\[\mc I_\la(u)\geq \frac{1}{2}\|u\|^2  -\frac{\la}{q+1}C_4 \|u\|^{q+1}-\frac{1}{22^*_{\al, m}}C_5\|u\|^{22^*_{\al, m}}.\]
Since $2<q+1<2^*_{m}<{22^*_{\al, m}}$, we can choose some $\rho>0$ small enough such that $\mc I_\la(u)\geq\mu>0$ for all $\|u\|=\rho$.
This complete the proof of $(1)$.\\
For $w\in\mb{H}_0^{m}(\Om)\setminus\{0\}$, we have
\[\mc I_\la(tw)= \frac{t^2}{2}\int_{\Om}|D^{m}w|^2 dx -\frac{\la t^{q+1}}{q+1}\int_{\Om}|w|^{q+1} dx-\frac{t^{22^*_{\al, m}}}{22^*_{\al, m}}\int_{\Om}\int_{\Om} \frac{|w(x)|^{2_{\al, m}^{*}}|w(y)|^{2_{\al,m}^{*}}}{|x-y|^{\al}}\,dx \,dy <0,\]
for $t>0$ large enough. Then we observe that $\mc I_\la(tw)<0$ by taking $e:=t_1w$ for large $t_1>0$, (2) follows.
\end{proof}
\begin{lemma}\label{l2}
Let $1<q<2^*_{m}-1$, $\la>0$. If $\{u_n\}$ is a $(PS)_c$ sequence of $\mc{I_\la}$, then
\begin{enumerate}
    \item $\{u_n\}$ is bounded.
    \item There exist $u_0\in \mb{H}^m_0(\Om)$  such that $u_0$  is a weak solution of the above problem \eqref{e5}.
    \item $\mc{I}_\la(u_0)\geq 0.$
\end{enumerate}\end{lemma}
\begin{proof} Let $\{u_n\}$ is a $(PS)_c$ sequence of $\mc{I_\la}$. Then
 there exists $C^{\prime} >0$, such that\[\left|\mc{I_\la}(u_n)\right|\leq C^{\prime}\quad and \quad \left|\left\langle \mc I_\la^{\prime}( u_n),\frac{u_n }{\| u_n\|}\right\rangle\right|\leq C^{\prime}.\]
\begin{enumerate}
 \item  Consider
\begin{align*}
C^{\prime}(1+\|u_n\|)& \geq  \mc{I_\la}(u_n)- \frac{1}{q+1}\langle \mc I_\la^{\prime}(u_n),{u_n }\rangle\\
&=\left({\frac{1}{2}}-{\frac{1}{q+1}}\right)\int_{\Om}|D^{m}u_n|^2 dx + \left({\frac{1}{q+1}}-\frac{1}{22^*_{\al, m}}\right)\int_{\Om}\int_{\Om} \frac{|u_n(x)|^{2_{\al, m}^{*}}|u_n(y)|^{2_{\al,m}^{*}}}{|x-y|^{\al}} dx \\ 
&\geq \left({\frac{1}{2}}-{\frac{1}{q+1}}\right)\|u_n\|^2.\\
\end{align*}
Since $2<q+1<{2^*_{\al, m}}$. Thus $\{u_n\}$ is bounded in $\mb{H}^m_0(\Om)$.\\
\item Upto a subsequence, there exists $u_0\in\mb{H}_0^{m}(\Om)$ such that $u_n \weak u_0$ in $\mb{H}_0^{m}(\Om)$ and $u_n \weak u_0$ in $L^{2^*_m}(\Om)$ as $n\to+\infty$. Then
\[|u_n|^{2_{\al, m}^{*}}\weak|u_0|^{2_{\al, m}^{*}} \text{ in }   L^\frac{2N}{2N-\al}(\Om),\]\\
as $n\to+\infty$. Using the Hardy-Littlewood-Sobolev inequality, the Riesz potential defines a linear continuous map from $L^\frac{2N}{2N-\al}(\Om)$ to $L^\frac{2N}{\al}(\Om)$, hence
\[|x|^{-\al}*|u_n|^{2_{\al, m}^{*}} \weak|x|^{-\al}*|u_0|^{2_{\al, m}^{*}}\text{  weakly in }  L^\frac{2N}{\al}(\Om),\]\\
as $n\to\infty$. Now, combining this with the fact that
\[|u_n|^{2_{\al, m}^{*}-2}u_n\weak|u_0|^{2_{\al, m}^{*}-2}u_0\, \text{weakly in } L^\frac{2N}{N+2m-\al}(\Om),\]\\
as $n\to\infty$, we obtain
\[(|x|^{-\al}*|u_n|^{2_{\al, m}^{*}})|u_n|^{2_{\al, m}^{*}-2}u_n \weak
(|x|^{-\al}*|u_0|^{2_{\al, m}^{*}})|u_0|^{2_{\al, m}^{*}-2}u_0 \text{ weakly in }  L^\frac{2N}{N+2}(\Om),\]\\
as $n\to\infty$. Since for any $\phi\in C_c^{\infty}(\Om)$,
\[0\leftarrow\langle \mc I_\la^{\prime}(u_n), \phi\rangle=\int_{\Om}D^{m}u_n D^{m}\phi dx -\la\int_{\Om}|u_n|^{q-2}u_n{\phi} dx-\int_{\Om}\int_{\Om} \frac{|u_n(y)|^{2_{\al, m}^{*}}|u_n(x)|^{2_{\al, m}^{*}-2}u_n(x){\phi(x)}}{|x-y|^{\al}} \, dx \, dy.\]\\
\noi Taking the limit as $n\to\infty$, we obtain
\[\int_{\Om}D^{m}u_0 D^{m}\phi dx -\la\int_{\Om}|u_0|^{q-2}u_{0}{\phi} dx-\int_{\Om}\int_{\Om} \frac{|u_0(y)|^{2_{\al, m}^{*}}|u_0(x)|^{2_{\al, m}^{*}-2}u_0(x){\phi(x)}}{|x-y|^{\al}} \, dx \, dy=0,\]\\
for any $\phi\in\mb{H}_0^{m}(\Om)$, which means $u_0$ is a weak solution of the problem \eqref{e5}.\\
\item Finally, we take a test function $\phi=u_0\in\mb{H}_0^{m}(\Om)$ in \eqref{e5}, we have
\[\int_{\Om}|D^{m}u_0|^2dx =\la\int_{\Om}|u_0|^{q+1} dx+\int_{\Om}\int_{\Om} \frac{|u_0(x)|^{2_{\al, m}^{*}}|u_n(y)|^{2_{\al, m}^{*}}}{|x-y|^{\al}} \, dx \, dy\]
and so,
\[\mc{I_\la}(u_0)=\left({\frac{\la}{2}}-{\frac{\la}{q+1}}\right)\int_{\Om}|u_0|^{q+1} dx + \frac{N+2m-\al}{4N-2\al}\int_{\Om}\int_{\Om} \frac{|u_0(x)|^{2_{\al, m}^{*}}|u_0(y)|^{2_{\al,m}^{*}}}{|x-y|^{\al}} \, dx \, dy\geq0.\]
\end{enumerate}
Which completes the proof of Lemma \ref{l2}.
\end{proof}

\noindent In the next lemma, we prove a convergence criterion for the $(PS)_c$ sequence below a critical level to prove the main result.\\
\begin{lemma}\label{l4}
Let $\{u_n\}$ be a $(PS)_c$ sequence of $\mc{I_\la}$ with
$c<\frac{N+2m-\al}{4N-2\al} S_{H, L}^{\frac{2N-\al}{N+2m-\al}}$. Then $\{u_n\}$ has a convergent subsequence.
\end{lemma}
\begin{proof}
Let $u_0$ be the weak limit of $\{u_n\}$ that we have already obtained in Lemma \ref{l2}. We define $v_n:=u_n-u_0$, then we know $v_n \weak 0$ in $\mb{H}_0^m(\Om)$, $v_n \to 0$ in $L^s(\Om)$, where $1<s<2_m^*$ and $v_n \to 0$ a.e in $\Om$. Moreover, by the Brezis--Lieb Lemma in \cite{3brezis1983relation} and Lemma \ref{l3}, we know
\begin{align*}
&\int_{\Om}\left|D^m u_n\right|^2 d x=\int_{\Om}\left|D^mv_n\right|^2 d x+\int_{\Om}\left|D^m u_0\right|^2 d x+o_n(1), \\
&\int_{\Om}\left|u_n\right|^{q+1} d x=\int_{\Om}\left|v_n\right|^{q+1} d x+\int_{\Om}\left|u_0\right|^{q+1} d x+o_n(1)
\end{align*}
and
\begin{align*}
\int_{\Om} \int_{\Om} \frac{\left|u_n(x)\right|^{2_{\al,m}^{*}}\left|u_n(y)\right|^{2_{\al,m}^{*}}}{|x-y|^\al}\, dx \, dy=\int_{\Om} \int_{\Om} \frac{\left|v_n(x)\right|^{2_{\al, m}^{*}}\left|v_n(y)\right|^{2_{\al, m}^{*}}}{|x-y|^\al} \, dx \, dy+\int_{\Om} \int_{\Om} \frac{\left|u_0(x)\right|^{2_{\al, m}^{*}}\left|u_0(y)\right|^{2_{\al, m}^{*}}}{|x-y|^\al} \, dx \, dy+o_n(1).\end{align*}
Consequently, we have
\begin{align}\label{e36}
c = \lim _{n \to \infty}\mc{I}_\la(u_n)= & \frac{1}{2} \int_{\Om}|D^m u_n|^2 d x-\frac{\la}{q+1} \int_{\Om} |u_n|^{q+1} d x-\frac{1}{{2}{2_{\al, m}^*}} \int_{\Om} \int_{\Om} \frac{\left|u_n(x)\right|^{2_{\al, m}^*}\left|u_n(y)\right|^{2_{\al, m}^*}}{|x-y|^\al} \, dx \, dy\notag \\
= & \frac{1}{2} \int_{\Om}\left|D^mv_n\right|^2 d x-\frac{\la}{q+1} \int_{\Om} |v_n|^{q+1} d x+\frac{1}{2} \int_{\Om}\left|D^m u_0\right|^2 d x-\frac{\la}{q+1} \int_{\Om} |u_0|^{q+1} d x \notag\\
&  \quad -\frac{1}{{2}{2_{\al, m}^*}} \int_{\Om} \int_{\Om} \frac{\left|v_n(x)\right|^{2_{\al, m}^*}\left|v_n(y)\right|^{2_{\al, m}^*}}{|x-y|^\al} \, dx \, dy-\frac{1}{22_{\al, m}^*} \int_{\Om} \int_{\Om} \frac{\left|u_0(x)\right|^{2_{\al, m}^*}\left|u_0(y)\right|^{2_{\al, m}^*}}{|x-y|^\al} \, dx \, dy +o_n(1)\notag \\
= & \mc{I}_\la(u_0)+\frac{1}{2} \int_{\Om}\left|D^m v_n\right|^2 d x-\frac{\la}{q+1} \int_{\Om} |v_n|^{q+1} d x-\frac{1}{22_{\al, m}^*} \int_{\Om} \int_{\Om} \frac{|v_n(x)|^{2_{\al, m}^*}|v_n(y)|^{2_{\al, m}^*}}{|x-y|^\al} \, dx \, dy+o_n(1)\notag \\
\geq & \frac{1}{2} \int_{\Om}\left|D^m v_n\right|^2 d x-\frac{1}{22_{\al, m}^*} \int_{\Om} \int_{\Om} \frac{|v_n(x)|^{2_{\al, m}^*}|v_n(y)|^{2_{\al, m}^*}}{|x-y|^\al} \, dx \, dy+o_n(1),
\end{align}
since $\mc{I_\la}(u_0) \geq 0$ and $\int_{\Om} |v_n|^{q+1} d x \rightarrow 0$, as $n \rightarrow + \infty$. Similarly, since $\left\langle \mc{I_\la}^{\prime}(u_0), u_0\right\rangle=0$,
we have
\begin{align}\label{e37}
o_n(1)= & \left\langle \mc{I}_\la^{\prime}\left(u_n\right), u_n\right\rangle -\left\langle \mc{I}_\la^{\prime}\left(u_0\right), u_0\right\rangle\notag\\
= &\int_{\Om}\left|D^m (u_n -u_0)\right|^2 d x-\la \int_{\Om} |u_n -u_0|^{q+1} d x-\int_{\Om} \int_{\Om} \frac{\left|u_n(x)-u_0(x)\right|^{2_{\al,m}^{*}}\left|u_n(y)-u_0(y)\right|^{2_{\al,m}^{*}}}{|x-y|^\al} \, dx \, dy+o_n(1)\notag \\
= & \int_{\Om}\left|D^m v_n\right|^2 d x-\int_{\Om} \int_{\Om} \frac{\left|v_n(x)\right|^{2_{\al,m}^{*}}\left|v_n(y)\right|^{2_{\al,m}^{*}}}{|x-y|^\al} \, dx \, dy+o_n(1).
\end{align}
From \eqref{e37}, we know there exists a non-negative constant ${\beta}$ such that
\[\int_{\Om}\left|D^m v_n\right|^2 d x \rightarrow {\beta}\]
and
\[\int_{\Om} \int_{\Om} \frac{\left|v_n(x)\right|^{2_{\al, m}^*}\left|v_n(y)\right|^{2_{\al, m}^*}}{|x-y|^\al} \, dx \, dy \rightarrow {\beta},\]
as $n \rightarrow+\infty$. Thus from \eqref{e36}, we obtain
\begin{equation}\label{e8}
c \geq \frac{N+2m-\al}{4N-2\al} {\beta}.
\end{equation}
By the definition of the best constant $S_{H, L}$ in \eqref{e4}, we have
\[S_{H, L}\left(\int_{\Om} \int_{\Om} \frac{\left|v_n(x)\right|^{2_{\al, m}^*}\left|v_n(y)\right|^{2_{\al, m}^{*}}}{|x-y|^\al} \, dx \, dy\right)^{\frac{N-2m}{2N-\al}} \leq \int_{\Om}\left|D^m v_n\right|^2 d x,\]
which yields ${\beta} \geq S_{H, L} {\beta}^{\frac{N-2m}{2N-\al}}$. Thus we have either ${\beta}=0$ or ${\beta} \geq S_{H, L}^{\frac{2N-\al}{N+2m-\al}}$. If ${\beta}\geq S_{H, L}^{\frac{2N-\al}{N+2m-\al}}$, then we obtain from \eqref{e8} that
\[\frac{N+2m-\al}{4N-2\al} S_{H, L}^{\frac{2N-\al}{N+2m-\al}} \leq \frac{N+2m-\al}{4N-2\al} {\beta} \leq c,\]
which contradicts with the fact that $c<\frac{N+2m-\al}{4 N-2 \al} S_{H, L}^{\frac{2N-\al}{N+2m-\al}}$. Thus $\beta=0$, and
\[\left\|u_n-u_0\right\| \rightarrow 0,\]
as $n \rightarrow+\infty$.
\end{proof}
\begin{lemma}\label{l5}
There exists $u_{\eps}$ such that
\[\sup _{t \geq 0} \mc{I_\la}(t u_\eps)<\frac{N+2m-\al}{4N-2\al} S_{H, L}^{\frac{2N-\al}{N+2m-\al}},\]
provided that either
\begin{enumerate}
\item  $N>\max \left\{\min \left\{\frac{2m(q+3)}{q+1}, 2m+\frac{\al}{q+1}\right\}, \frac{2m(q+1)}{q}\right\}$ and $\la>0$, or
\item  $N \leq \max \left\{\min \left\{\frac{2m(q+3)}{q+1}, 2m+\frac{\al}{q+1}\right\}, \frac{2m(q+1)}{q}\right\}$ and $\la$ is sufficiently large.
\end{enumerate}
\end{lemma}
\begin{proof}

From equation \eqref{e11} and \eqref{e52}, we know that
\begin{align*}
\int_{\Om}\left|D^m u_{\eps}\right|^2 dx = C(N, \al)^{\frac{N(N-2m)}{2m(2N-\al)}} S_{H, L}^{\frac{N}{2m}}+O\left(\eps^{N-2m}\right)
\end{align*}
and
\begin{align}\label{e12}
\int_{\Om} \int_{\Om} \frac{\left|u_{\eps}(x)\right|^{2_{\al, m}^*}\left|u_{\eps}(y)\right|^{2_{\al, m}^*}}{|x-y|^\al} \, dx \, dy \geq C(N, \al)^{\frac{N}{2m}} S_{H, L}^{\frac{2N-\al}{2m}}-O\left(\eps^{\frac{2N-\al}{2}}\right).
\end{align}\\
\item[Case 1].
$N>\max \left\{\min \left\{\frac{2m(q+3)}{q+1}, 2m+\frac{\al}{q+1}\right\}, \frac{2m(q+1)}{q}\right\}$.\\
Since $q>1$ and $N>\frac{2m(q+1)}{q}$ we know $N<(N-2m)(q+1)$, thus we have
\begin{align}\label{e13}
\int_{\Om} u_{\eps}^{q+1} dx \geq O\left(\eps^{N-\frac{(N-2m)(q+1)}{2}}\right),\end{align}
for $\eps>0$ sufficiently small.\\
Now, employing the estimates  \eqref{e11}, \eqref{e12} and \eqref{e13} in $\mc{I_\la}$, we obtain
\begin{align*}
\mc{I_\la}\left(t u_{\eps}\right)= & \frac{t^2}{2} \int_{\Om}\left|D^m u_{\eps}\right|^2 d x-\la \frac{t^{q+1}}{q+1} \int_{\Om}\left|u_{\eps}\right|^{q+1} d x-\frac{t^{22_{\al, m}^*}}{22_{\al, m}^*} \int_{\Om} \int_{\Om} \frac{\left|u_{\eps}(x)\right|^{2_{\al, m}^*}\left|u_{\eps}(y)\right|^{2_{\al, m}^*}}{|x-y|^\al} \, dx \, dy \\
\leq & \frac{t^2}{2}\left(C(N, \al)^{\frac{N(N-2m)}{2m(2N-\al)}} S_{H, L}^{\frac{N}{2m}}+O\left(\eps^{N-2m}\right)\right)-\la\frac{t^{q+1}}{q+1} O\left(\eps^{N-\frac{(N-2m)(q+1)}{2}}\right) \\
& -\frac{t^{22_{\al, m}^*}}{22_{\al, m}^*}\left(C(N, \al)^{\frac{N}{2m}} S_{H, L}^{\frac{2N-\al}{2m}}-O\left(\eps^{N-\frac{\al}{2}}\right)\right) \\
:= &  K(t).
\end{align*}
Clearly, $ K(t) \rightarrow-\infty$ as $t \rightarrow+\infty$ and $K(t)>0$ as t is small enough. It gives us that there exists $t_{\eps}>0$ such that 
$\displaystyle \sup_{t \geq 0}K(t)$ is attained at $t_{\eps}$. Differentiating $K(t)$ and equating it to zero, we obtain that
\[t_{\eps}\left(C(N, \al)^{\frac{N(N-2m)}{2m(2N-\al)}} S_{H, L}^{\frac{N}{2m}}+O\left(\eps^{N-2m}\right)\right)-\la t_\eps^q O\left(\eps^{N-\frac{(N-2m)(q+1)}{2}}\right)-t_\eps^{22_{\al, m}^*-1}\left(C(N, \al)^{\frac{N}{2m}} S_{H, L}^{\frac{2N-\al}{2m}}-O\left(\eps^{N-\frac{\al}{2}}\right)\right)=0
\]
and so
\[
t_\eps<\left(\frac{C(N, \al)^{\frac{N(N-2m)}{2m(2N-\al)}} S_{H, L}^{\frac{N}{2m}}+O\left(\eps^{N-2m}\right)}{C(N, \al)^{\frac{N}{2m}} S_{H, L}^{\frac{2N-\al}{2m}}-O\left(\eps^{N-\frac{\al}{2}}\right)}\right)^{\frac{1}{22_{\al, m}^*-2}}:=\mc{T}_\eps,
\]
and there exists $t_0>0$ independent of $\eps$ such that for $\eps>0$ small enough
\[t_\eps>t_0.\]
We note that the function
\[
t \mapsto \frac{t^2}{2}\left(C(N, \al)^{\frac{N(N-2m)}{2m(2N-\al)}} S_{H, L}^{\frac{N}{2m}}+O\left(\eps^{N-2m}\right)\right)-\frac{t^{22_{\al, m}^*}}{22_{\al, m}^*}\left(C(N, \al)^{\frac{N}{2m}} S_{H, L}^{\frac{2N-\al}{2m}}-O\left(\eps^{N-\frac{\al}{2}}\right)\right),
\]
is increasing on $\left[0, \mc{T}_\eps\right]$, we have
\begin{align*}
\max _{t \geq 0} \mc{I_\la}\left(t u_{\eps}\right) & \leq \frac{N+2m-\al}{4 N-2 \al}\left(\frac{C(N, \al)^{\frac{N(N-2m)}{2m(2N-\al)}} S_{H, L}^{\frac{N}{2m}}+O\left(\eps^{N-2m}\right)}{\left(C(N, \al)^{\frac{N}{2m}} S_{H, L}^{\frac{2N-\al}{2m}}-O\left(\eps^{N-\frac{\al}{2}}\right)\right)^{\frac{N-2m}{2N-\al}}}\right)^{\frac{2N-\al}{N+2m-\al}}-O\left(\eps^{N-\frac{(N-2m)(q+1)}{2}}\right)\ \\
& \leq \frac{N+2m-\al}{4N-2\al} S_{H, L}^{\frac{2N-\al}{N+2m-\al}}+O\left(\eps^{\min \left\{N-2m, {N-\frac{\al}{2}}\right\}}\right)-O\left(\eps^{N-\frac{(N-2m)(q+1)}{2}}\right) \\
& <\frac{N+2m-\al}{4N-2\al} S_{H, L}^{\frac{2N-\al}{N+2m-\al}},
\end{align*}
thanks to $t_0<t_{\eps}<\mc{T}_\eps$, \eqref{e13} and $N>\min \left\{\frac{2m(q+3)}{q+1}, 2m+\frac{\al}{q+1}\right\}$.
\item[Case 2]. 
$N \leq \max \left\{\min \left\{\frac{2m(q+3)}{q+1}, 2m+\frac{\al}{q+1}\right\}, \frac{2m(q+1)}{q}\right\}$.
For any fixed $\eps$ in \eqref{e10}, we have $\mc{I_\la}\left(t u_{\eps}\right) \rightarrow-\infty,$ as $t \rightarrow+\infty$. 
Therefore, there exist $t_\la>0$ such that
$\DD\max _{t \geq 0} \mc{I_\la}\left(t u_{\eps}\right)=\mc{I_\la}\left(t_{\la} u_{\eps}\right).$
Moreover $\frac{d \mc{I_\la}(t u_\eps)}{d t}=0$ for $t_\la>0$ and $t_\la$ satisfies
\[t_\la \int_{\Om}\left|D^m u_{\eps}\right|^2 d x=\la t_\la^q \int_{\Om}\left|u_{\eps}\right|^{q+1} d x+t_\la^{22_{\al, m}^*-1} \int_{\Om} \int_{\Om} \frac{\left|u_{\eps}(x)\right|^{2_{\al, m}^*}\left|u_{\eps}(y)\right|^{2_{\al, m}^*}}{|x-y|^\al} \, dx \, dy, \]
which implies that $t_\la \rightarrow 0$ as $\la \rightarrow+\infty$. Therefore,
\[\max _{t \geq 0} \mc{I_\la}\left(t u_{\eps}\right)=\frac{t_\la^2}{2} \int_{\Om}\left|D^m u_{\eps}\right|^2 dx-\frac{\la t_\la^{q+1}}{q+1} \int_{\Om}\left|u_{\eps}\right|^{q+1} dx-\frac{t_\la^{22_{\al, m}^*}}{22_{\al, m}^*} \int_{\Om} \int_{\Om} \frac{\left|u_{\eps}(x)\right|^{2_{\al, m}^*}\left|u_{\eps}(y)\right|^{2_{\al, m}^*}}{|x-y|^\al} \, dx \, dy \rightarrow 0,\]
as $\la \rightarrow+\infty$, which gives the desired conclusion for the case $N \leq \max \left\{\min \left\{\frac{2m(q+3)}{q+1}, 2m+\frac{\al}{q+1}\right\}, \frac{2m(q+1)}{q}\right\}$.
\end{proof}
{\bf{\noi Proof of Theorem \ref{t1}:}}
By Lemma \ref{l1} and the Mountain Pass Theorem without $(PS)$ condition  \cite{4willem1996minimax}, there exists a $(PS)$ sequence $\{u_n\}$ such that $\mc{I_\la}(u_n) \rightarrow c$ and $\mc{I^{\prime}_\la}(u_n) \rightarrow 0$ in $\mb{H}_0^m(\Om)^{-1}$ at the minimax level
\[c=\inf _{\ga \in \Gamma} \max _{t \in[0,1]} \mc{I_\la}(\gamma(t))>0,\]
where
\[\Gamma:=\left\{\gamma \in C\left([0,1], \mb{H}_0^m(\Om)\right): \gamma(0)=0, \mc{I_\la}(\gamma(1))<0\right\}.\]
From Lemma \ref{l5} and the definition of $c$, we know
$c<\frac{N+2m-\al}{4N-2 \al} S_{H, L}^{\frac{2N-\al}{N+2m-\al}}$, provided that either
\begin{enumerate}
\item $N>\max \left\{\min \left\{\frac{2m(q+3)}{q+1}, 2m+\frac{\al}{q+1}\right\}, \frac{2m(q+1)}{q}\right\}$ and $\la>0$, or
\item $N \leq \max \left\{\min \left\{\frac{2m(q+3)}{q+1}, 2m+\frac{\al}{q+1}\right\}, \frac{2m(q+1)}{q}\right\}$ and $\la$ is sufficiently large.
\end{enumerate}
Applying Lemma \ref{l4}, we know $\{u_n\}$ contains a convergent subsequence. We have $\mc{I_\la}$ has a critical value 
$c \in\left(0, \frac{N+2m-\al}{4N-2\al} S_{H, L}^{\frac{2N-\al}{N+2m-\al}}\right)$ and thus the problem \eqref{e5} has a nontrivial solution.   \qed
\section{\bf{\texorpdfstring{$q=1$}{}}}
\noi In this section, we give details of the proof of Theorem \ref{t4} and the existence result depends on the dimension $N$ and on the parameter $\la$.
Further, this section is divided into two sub-sections. In the first sub-section we discuss for $N \geq 2m+1$, $0<{\la}<{\la}^m_1$ and the case $N \geq 2m+1$, ${\la}>{\la}^m_1$ is discuss in second subsection. Before starting the subsections, we define the energy functional $\mc{I_\la}:\mb{H}_0^{m}(\Om)\to \mb{R}^N$ corresponding to the problem \eqref{e5} as
\[\mc{I_\la}(u)= \frac{1}{2}\int_{\Om}|D^{m}u|^2 dx -\frac{\la}{2}\int_{\Om}|u|^2 dx-\frac{1}{22^*_{\al, m}}\int_{\Om}\int_{\Om} \frac{|u(x)|^{2_{\al, m}^{*}}|u(y)|^{2_{\al, m}^{*}}}{|x-y|^{\al}} \, dx \, dy.\]
\subsection{\bf{The case \texorpdfstring{$N \geq 2m+1$, $0<{\la}<{\la}^m_1$}{N >= 2m+1, 0<lambda<lambda^m_1}{}}}
\noi In this sub-section, we proved that the problem \eqref{e5} admits a nontrivial weak solution.
\begin{lemma}\label{l18}
Let $N$ be the dimension and $\la>0$ be the position parameter, then, there exist $u_{\eps} \in H_0^m(\Om) \backslash\{0\}$ such that
\[\frac{\|u_{\eps}\|_2^2-\la|u_{\eps}|_2^2}{\|u_{\eps}\|_{N L}^2}<S_{H, L} \quad \text{for}  \ {\eps>0},\] provided that:
\begin{enumerate}
\item $\la>\la^*>0$ whenever $N \in[2m+1, 4m)$;
\item  $\la>0$ whenever $N\geq4m$.
\end{enumerate}
\end{lemma}
\begin{proof}
If $N \in[2m+1, 4m)$, then using \eqref{e54}, \eqref{e11} and \eqref{e12}, yields

\begin{align*}
\frac{\|u_{\eps}\|_2^2-\la\left|u_{\eps}\right|_2^2}{\left\|u_{\eps}\right\|_{N L}^2} & \leq \frac{C(N, \al)^{\frac{N(N-2m)}{2m(2N- \al)}} S_{H, L}^{\frac{N}{2m}}+ O\left(\eps^{N-2m}\right)-\la C_1 \eps^{N-2m} +O\left(\eps^{2m} \right)}{\left(C(N, \al)^{\frac{N}{2m}} S_{H, L}^{\frac{2N-\al}{2m}}-O\left(\eps^{N-\frac{\al}{2}}\right)\right)^{\frac{N-2m}{2N-\al}}} \\
&=S_{H, L} +O\left(\eps^{\frac{2N-\al}{2}}\right)+O\left(\eps^{N-2m}\right) -\la C_1 \eps^{N-2m} \\
&=S_{H, L} +O\left(\eps^{min\left\{N-2m, N-{\frac{\al}{2}}\right\}}\right) -\la C_1 \eps^{N-2m}\\ 
&<S_{H, L},
\end{align*}
if $\la$ is large enough, i.e $\la>\la^*>0$.\\
If $N=4m$, we have
\begin{align*}
\frac{\|u_{\eps}|_2^2-\la|u_{\eps}|_2^2}{\|u_{\eps}\|_{N L}^2} & \leq\frac{C(N, \al)^{\frac{N(N-2m)}{2m(2N-\al)}} S_{H, L}^{\frac{N}{2m}}+O\left(\eps^{N-2m}\right)-\la C_2 |ln{\eps}| \eps^{2m} -O\left(\eps^{2m}\right)}{\left(C(N, \al)^{\frac{N}{2m}} S_{H, L}^{\frac{2N-\al}{2m}}-O\left(\eps^{N-\frac{\al}{2}}\right)\right)^{\frac{N-2m}{2N-\al}}} \\
& =S_{H, L}-\frac{\la C_2 \eps^2|\ln \eps|}{\left(C(4m, \al)^2 S_{H, L}^{\frac{8m-\al}{2m}}-O\left(\eps^{4m-\frac{\al}{2}}\right)\right)^{\frac{2m}{8m-\al}}}+O\left(\eps^{2m}\right) \\
& \leq S_{H, L}-\la C_2 \eps^2|\ln \eps|+O\left(\eps^{2m}\right) \\
& <S_{H, L}.
\end{align*}
Similarly, if $N> 4m$, we have
\begin{align*}
\frac{\|u_{\eps}|_2^2-\la|u_{\eps}|_2^2}{\|u_{\eps}\|_{N L}^2} & \leq\frac{C(N, \al)^{\frac{N(N-2m)}{2m(2N-\al)}} S_{H, L}^{\frac{N}{2m}}+O\left(\eps^{N-2m}\right)-\la C_3\eps^{2m} -O\left(\eps^{N-2m}\right)}{\left(C(N, \al)^{\frac{N}{2m}} S_{H, L}^{\frac{2N-\al}{2m}}-O\left(\eps^{N-\frac{\al}{2}}\right)\right)^{\frac{N-2m}{2N-\al}}} \\
& \leq S_{H, L}-\la C_3 \eps^{2m} +O\left(\eps^{N-2m}\right)\\
&<S_{H, L}.
\end{align*}
This completes the proof.
\end{proof}
\noi Now, we show that $\mc{I}_\la$ has the geometric structure of the Mountain Pass Theorem when $\la \in (0, \la^m_1 )$.
\begin{lemma} \label{l19}
If $N\geq 2m+1$ and ${\la}\in (0, \la^m_1)$, then the functional   $\mc{I}_{\la}$ satisfies the following properties:
\begin{enumerate}
\item There exist $\mu, \rho>0$ such that $\mc{I}_\la(u)\geq \mu$ for $\|u\|=\rho$.
\item There exist $e\in \mb{H}^m_0(\Om)$ with $\|e\|>{\rho}$ such that $\mc{I}_\la(e)<0$.   
\end{enumerate}\end{lemma}

\begin{proof}\begin{enumerate}
\item Using the Sobolev embedding and Hardy-Littlewood-Sobolev inequality, for all $u\in\mb{H}_0^{m}(\Om)\setminus\{0\}$ and ${\la}\in (0,\la^m_1)$, we obtain
\begin{align*}
\mc I_\la(u) &\geq \frac{1}{2}\| u \|^2  -\frac{\la} {2}\int_{\Om}|u|^{2} dx-\frac{1}{22^*_{\al, m}}C(N, \al)S^{-2^*_{\al, m}}\| u \|^{22^*_{\al, m}}\\
&\geq\frac{1}{2}\left(1-\frac{\la}{\la^m_1}\right)\| u \|^2 -\frac{C_5}{22^*_{\al, m}}\| u \|^{22^*_{\al, m}}.
\end{align*}
Since $2<{22^*_{\al, m}}$, we can choose some $\rho$ small enough such that $\mc I_\la(u)\geq\mu>0$ for all u that satisfy $\|u\|=\rho$.
\item For some  $v\in\mb{H}_0^{m}(\Om)\setminus\{0\}$, we have
\[\mc I_\la(tv)= \frac{t^2}{2}\int_{\Om}|D^{m}v|^2 dx -\frac{\la t^{2}}{2}\int_{\Om}|v|^{2} dx-\frac{t^{22^*_{\al, m}}}{22^*_{\al, m}}\int_{\Om}\int_{\Om} \frac{|v(x)|^{2_{\al, m}^{*}}|v(y)|^{2_{\al, m}^{*}}}{|x-y|^{\al}} \, dx \, dy <0\]
for $t>0$ large enough. Choosing $e:=t_1v$ for some large $t_1>0$ and the conclusion (2) follows. 
\end{enumerate}
\end{proof}
\noi{\bf{Remark}}: $\mc I_\la$ satisfy mountain pass geometry. Therefore, there exists a  Palais-Smale $(PS)$ sequence $\{u_n\}$ such that $\mc{I}_\la(u_n)\rightarrow c$ and $\mc{I^{\prime}_\la}(u_n)\rightarrow 0$ in $\mb{H}_0^{m}(\Om)^{-1}$ at the  minimax level 
\[c=\inf_{\ga\in\Ga}\max_{t\in[0, 1]} \mc{I}_\la(\ga(t))>0,\]
where
\[\Ga:=\{\ga\in C[0, 1], \mb{H}_0^{m}(\Om):\ga(0)=0, \mc{I}_\la(\ga(1))<0\}.\]

\begin{proposition} \label{p3}
Let $N\geq2m+1$, $0<\al<N$ and $c>0$ be the critical value of $\mc{I_\la}$ such that
$c<\frac{N+2m-\al}{4N-2\al} S_{H, L}^{\frac{2N-\al}{N+2m-\al}}$.
\end{proposition}
\begin{proof}
From Lemma \ref{l18}, we know there exists $u_{\eps}\in \mb{H}_0^{m}(\Om)\setminus\{0\}$ such that 
\[\frac{\int|D^m u_{\eps}|^2 dx-\la |u_{\eps}|^2_2}{\|u_{\eps}\|^2_{NL}}<S_{H, L}.\]
Therefore,
\begin{align*}
0<\max_{t\geq0}\mc{I}_\la(tu_{\eps})&=\max_{t\geq0}\left(\frac{t^2}{2}\int_{\Om}|D^m u_{\eps}|^2dx-\frac{\la t^2}{2}\int_{\Om}u_{\eps}^2dx-\frac{t^{22^*_{\al, m}}}{22^*_{\al, m}}\int_\Om\int_\Om\frac{|u_{\eps}(x)|^{2^*_{\al, m}}|u_{\eps}(y)|^{2^*_{\al, m}}}{|x-y|^\al} \, dx \, dy\right)\\
&=\frac{N+2m-\al}{4N-2\al}\left(\frac{\DD \int|D^m u_{\eps}|^2 dx-\la |u_{\eps}|^2_2}{||u_{\eps}||^2_{NL}}\right)^\frac{2N-\al}{N+2m-\al}\\
&<\frac{N+2m-\al}{4N-2\al}S^\frac{2N-\al}{N+2m-\al}_{H, L}.
\end{align*}
By the definition of $c$, we know $c<\frac{N+2m-\al}{4N-2\al}S^\frac{2N-\al}{N+2m-\al}_{H, L}$.
\end{proof}
\vspace{0.2cm}
\noi {\bf{Proof of Theorem \ref{t4}-(1)}}
Let $\{u_n\}$ be the $(PS)_c$ sequence obtained in Proposition \ref{p3}. Lemma \ref{l4}, yield $\{u_n\}$ has a convergent subsequence. And so, we have $\mc{I}_\la$ has a critical value $c\in\left(0, \frac{N+2m-\al}{4N-2\al}S^\frac{2N-\al}{N+2m-\al}_{H,L}\right)$ and the problem \eqref{e5} has a nontrivial solution.\qed


\subsection{\bf{The case \texorpdfstring{$N \geq 2m+1, \la \geq \la^m_1$}{}}}  In this sub-section, we proved the existence of nontrivial solution for $\la\geq\la^m_1$ where $\la$ is not an eigenvalue, using the linking argument based on the spectral decomposition of $\mb{H}_0^{m}(\Om)$.
Moreover, for $N\in[2m+1, 4m)$, there exists $\la^{*}>0$ such that the problem \eqref{e5} has a solution for all $\la>\la^{*}$.

\vspace{0.2cm}

\noi Suppose that $\la \in\left[\la^m_j, \la^m_{j+1}\right)$ for some $j\in\mb{N}$, where $\la^m_j$ denotes the $j$-th eigenvalue of $(-\Delta)^m$ in $\Om$ having a smooth boundary and $e_j$ is the $j$-th eigenfunction corresponding to the eigenvalue $\la^m_j$.\\
We denote
\[E_{j+1} := \{u\in \mb{H}^m_0(\Om): \langle u, e_i \rangle_{\mb{H}^m_0(\Om)} = 0, \forall \,    i = 1, 2, . . . , j\}\]
and $W_j :=\operatorname{span}\{e_1, \ldots, e_j, \}$  denote the linear subspace generated by the first $j$ eigenfunctions of $(-\Delta)^m$
for any $j\in\mb{N}$. It is easily seen that $W_j$ is finite dimensional and $W_j \oplus E_{j+1} =\mb{H}^m_0(\Om)$.
\begin{lemma}\label{l10}
If $N\geq2m+1$ and $\la \in\left[\la^m_j, \la^m_{j+1}\right)$ for some $j \in \mb{N}$, then, the functional $\mc{I}_\la$ satisfies the following:
\begin{enumerate}
\item There exist $\mu, \rho>0$ such that for any $u \in \mb{E}_{j+1}$ with $\|u\|=\rho$ then $\mc{I}_\la(u) \geq \mu$.
\item  $\mc{I}_\la(u)<0$ for any $u \in {W}_j$.
\item  Let $F$ be a finite dimensional subspace of $\mb{H}_0^m(\Om)$. Then there exists $R>\rho$ such that for any $u \in F$ with $\|u\| \geq R$, $\mc{I}_\la(u) \leq 0$.
\end{enumerate}\end{lemma}
\begin{proof}
\begin{enumerate}\item Let $\la \in\left[\la^m_j, \la^m_{j+1}\right)$ and $u \in \mb{E}_{j+1} \backslash\{0\}$, then using Sobolev embedding and the Hardy-Littlewood-Sobolev inequality, we have
\begin{equation*}
\mc{I}_\la(u) \geq\frac{1}{2}\left(1-\frac{\la}{\la^m_{j+1}}\right)\| u \|^2 -\frac{C_5}{22^*_{\al, m}}\| u \|^{22^*_{\al, m}},
\end{equation*}
since $2<22_{\al, m}^*$. Choose $\mu$, $\rho>0$ such that $\mc{I}_\la(u) \geq \mu$ for any $u \in \mb{E}_{j+1}$ with $\|u\|=\rho$.
\item Let $u \in {W}_j$, then $u=\sum_{i=1}^j l_i e_i$, with $l_i \in \mb {R}, i=1, \ldots, j$. Since $\left\{e_i\right\}_{i \in \mb {~N}}$ is an orthonormal basis of $L^2(\Om)$ and an orthogonal basis of $\mb{H}_0^m(\Om)$, we have
\[
\int_{\Om} u^2 d x=\sum_{i=1}^j l_i^2 \quad \text { and } \quad \int_{\Om}|D^m u|^2 d x=\sum_{i=1}^j l_i^2\int_{\Om}|D^{m} e_i|^2
\]
which gives
\begin{align*}
\mc{I}_\la(u) & =\frac{1}{2} \sum_{i=1}^j l_i^2\left(\DD\int_{\Om}|D^{m}e_i|^2-\la\right)-\frac{1}{22_{\al, m}^*} \int_{\Om} \int_{\Om} \frac{|u(x)|^{2_{\al, m}^*}|u(y)|^{2_{\al, m}^*}}{|x-y|^\al} \, dx \, dy \\
& <\frac{1}{2} \sum_{i=1}^j l_i^2\left(\la^m_i-\la\right)\leq 0,
\end{align*}
since $\la^m_i\leq\la^m_j\leq\la$.
\item Let $u \in F \backslash\{0\}$, then using non-negativity of $\la$, we obtain
\begin{align*}
\mc{I}_\la(u) & \leq \frac{1}{2}\|u\|^2-\frac{1}{22_{\al, m}^*}\|u\|_{N L}^{22_{\al, m}^*} \\
& \leq \frac{1}{2}\|u\|^2-\frac{C_5}{22_{\al, m}^*}\|u\|^{22_{\al, m}^*},
\end{align*}
for some positive constant $C_5>0$, since on a finite-dimensional space all norms are equivalent. Therefore, $\mc{I}_\la(u) \rightarrow-\infty$ as $\|u\| \rightarrow+\infty$. So, there exists $R>\rho$ such that for any $u \in \mathbf{F}$ with $\|u\| \geq R$ we have $\mc{I}_\la(u) \leq 0$ which proves (3).\\
\vspace{0.01cm}

\noi Now define the linear space
\[\mathbf{P}_{j, \eps}:=\operatorname{span}\{e_1, \ldots, e_j, u_\eps\} \text { for any } j \in \mb{N}\]
and set
\[m_{j, \eps}:=\max _{u \in \mathbf{P}_{j, \eps}, \|u\|_{NL}=1}\left(\int_{\Om}|D^m u|^2 d x-\la \int_{\Om}|u|^2 d x\right),\]
where $\|\cdot\|_{NL}$ same as defined in Lemma \ref{l17}.
\end{enumerate}\end{proof}
\begin{lemma}
If $N \geq 2m+1$ and $\la \in[\la^m_j, \la^m_{j+1})$ for some $j \in \mb{N}$. Then
\begin{enumerate}
\item $m_{j,\eps}$ is achieved at some $u_m \in \mathbf{P}_{j, \eps}$ and $u_m=v+t u_\eps$\text{ with } $v \in {W}_j$ and $t \geq 0$.
\item The following estimate holds true
\begin{equation}\label{e21}
m_{j,\eps} \leq\left\{\begin{array}{l}
\left(\la^m_j-\la\right)|v|_2^2 \quad \text { if } t=0, \\
\left(\la^m_j-\la\right)|v|_2^2+A_{\eps}\left(1+|v|_2 O\left(\eps^{\frac{N-2m}{2}}\right)\right)+ O\left(\eps^{\frac{N-2m}{2}}\right)|v|_2 \quad \text { if } t>0,
\end{array}\right.
\end{equation}
as $\eps \rightarrow 0$, where  $v \in {W}_j$, $u_{\eps}$ is defined in Section $2$ and
\begin{equation}\label{e22}
A_{\eps}=\frac{\|u_{\eps}\|_2^2-\la|u_{\eps}|_2^2}{\|u_{\eps}\|_{NL}^2}.
\end{equation}\end{enumerate}
\end{lemma}
\begin{proof}
\noi \begin{enumerate}
\item As we know that $\mathbf{P}_{j, \eps}$ is a finite dimensional space and $m_{\eps}$ is achieved at some $u_m \in \mathbf{P}_{j, \eps}$, that is,
\[m_{j, \eps}=\|D^m u_m\|_2^2-\la\left|u_m\right|_2^2 \quad \text { and } \quad\left\|u_m\right\|_{NL}=1.\]
Evidently, $u_m \not \equiv 0$. The definition of $\mathbf{P}_{j, \eps}$ yields
$u_m=v+t u_{\eps}$ for some $v \in {W}_j$ and $t \in \mb{R}$. Assume that $t \geq 0$, otherwise, if $t<0$ we can replace $u_m$ with $-u_m$. The result follows.
\item If $t=0$, then $u_m=v \in {W}_j$ and
\[m_{j,\eps}=\|D^{m}u_m\|_2^2-\la\left|u_m\right|_2^2=|D^{m}v|_2^2-\la|v|_2^2 \leq\left(\la^m_j-\la\right)|v|_2^2.\]
Now, for the case $t>0$ we have $v \in L^{\infty}(\Om)$, Since $e_1, \ldots, e_j \in L^{\infty}(\Om)$, By a direct computation, we get
\begin{align*}
\int_{B_{l_0 r}} \int_{B_{l_0 r}} &\frac{|u_{\eps}(x)|^{2_{\al, m}^*}|u_{\eps}(y)|^{2_{\al, m}^*-1}}{|x-y|^\al} \, dx \, dy \\
& =\int_{B_{l_0r}} \int_{B_{l_0r}}  \frac{|U_{\eps}(x)|^{2_{\al, m}^*}|U_{\eps}(y)|^{2_{\al, m}^*-1}}{|x-y|^\al} \, dx \, dy \\
&=[\eps^{\frac{N-2m}{2}}C_{N,M}^{\frac{N-2m}{m}}]^{22_{\al, m}^*-1} \int_{B_{l_0r}} \int_ {B_{l_0r}}\frac{1}{({\eps}^2+|x|^2)^{\frac{2N-\al}{2}}|x-y|^\al({\eps}^2+|y|^2)^{\frac{N-\al+2m}{2}}} \, dx \, dy \\
&= O\left(\eps^{\frac{N-2m}{2}}\right) \int_{B_{\frac{l_0r}{\eps}}} \int_ {B_{\frac{l_0r}{\eps}}}\frac{1}{(1+|x|^2)^{\frac{2N-\al}{2}}|x-y|^\al(1+|y|^2)^{\frac{N-\al+2m}{2}}} \, dx \, dy \\ 
& \leq O\left(\eps^{\frac{N-2m}{2}}\right) \int_{\mb{R}^N} \int_{\mb{R}^N} \frac{1}{(1+|x|^2)^{\frac{2N-\al}{2}}|x-y|^\al(1+|y|^2)^{\frac{N-\al+2m}{2}}} \, dx \, dy. \\
\end{align*}
Using the Hardy--Littlewood--Sobolev inequality, we get
\begin{align*}
\int_{B_{l_0r}} \int_{B_{l_0r}}& \frac{|u_{\eps}(x)|^{2_{\al, m}^*}|u_{\eps}(y)|^{2_{\al, m}^*-1}}{|x-y|^\al}\, dx \, dy \\
&\leq O\left(\eps^{\frac{N-2m}{2}}\right) \left(\int_{\mb{R}^N} \frac{1}{\left(1+|x|^2\right)^N}dx\right)^{\frac{2N-\al}{2N}}\left(\int_{\mb{R}^N}\frac{1}{\left(1+|y|^2\right)^{\frac{N(N-\al+2m)}{2N-\al}}} d y \right)^{\frac{2N-\al}{2N}}\\
&=O\left(\eps^{\frac{N-2m}{2}}\right).
\end{align*}
Therefore, we obtain
\begin{equation} \label{e40}
\int_{\Om} \int_{\Om} \frac{\left|u_{\eps}(x)\right|^{2_{\al, m}^*}\left|u_{\eps}(y)\right|^{2_{\al, m}^*}-1}{|x-y|^\al} \, dx \, dy \leq O\left(\eps^{\frac{N-2m}{2}}\right).
\end{equation}
Alternatively, performing a direct computation gives
\begin{align*}
\int_{B_{r}} \int_{B_{r}} &\frac{|u_{\eps}(x)|^{2_{\al, m}^*}|u_{\eps}(y)|^{2_{\al, m}^*-1}}{|x-y|^\al} \, dx \, dy \\
& =\int_{B_{r}} \int_{B_{r}}  \frac{|U_{\eps}(x)|^{2_{\al, m}^*}|U_{\eps}(y)|^{2_{\al, m}^*-1}}{|x-y|^\al} \, dx \, dy \\
&=[\eps^{\frac{N-2m}{2}}C_{N,M}^{\frac{N-2m}{m}}]^{22_{\al, m}^*-1} \int_{B_{r}} \int_ {B_{r}}\frac{1}{({\eps}^2+|x|^2)^{\frac{2N-\al}{2}}|x-y|^\al({\eps}^2+|y|^2)^{\frac{N-\al+2m}{2}}} \, dx \, dy \\
&= O\left(\eps^{\frac{N-2m}{2}}\right) \int_{B_{\frac{r}{\eps}}} \int_ {B_{\frac{r}{\eps}}}\frac{1}{(1+|x|^2)^{\frac{2N-\al}{2}}|x-y|^\al(1+|y|^2)^{\frac{N-\al+2m}{2}}} \, dx \, dy \\ 
& \geq O\left(\eps^{\frac{N-2m}{2}}\right) \int_{B_{r}} \int_{B_{r}}
\frac{1}{(1+|x|^2)^{\frac{2N-\al}{2}}|x-y|^\al(1+|y|^2)^{\frac{N-\al+2m}{2}}} \, dx \, dy \\
&= O\left(\eps^{\frac{N-2m}{2}}\right).
\end{align*}
Thus,
\begin{equation} \label{e41}
\int_{\Om} \int_{\Om} \frac{\left|u_{\eps}(x)\right|^{2_{\al, m}^*}\left|u_{\eps}(y)\right|^{2_{\al, m}^*}-1}{|x-y|^\al} \, dx \, dy \geq O\left(\eps^{\frac{N-2m}{2}}\right).
\end{equation}
Hence, by \eqref{e40} and \eqref{e41}, we obtain
\[\int_{\Om} \int_{\Om} \frac{\left|u_{\eps}(x)\right|^{2_{\al, m}^*}\left|u_{\eps}(y)\right|^{2_{\al, m}^*}-1}{|x-y|^\al} \, dx \, dy =O\left(\eps^{\frac{N-2m}{2}}\right).\]
Applying convexity, we get
\begin{align} \label{e23} \notag
1= & \int_{\Om} \int_{\Om} \frac{\left|u_m(x)\right|^{2_{\al, m}^*}\left|u_m(y)\right|^{2_{\al, m}^*}}{|x-y|^\al} \, dx \, dy \\ \notag
= & \int_{\Om} \int_{\Om} \frac{\left|v(x)+t u_\eps(x)\right|^{2_{\al, m}^*}\left|v(y)+t u_\eps(y)\right|^{2_{\al, m}^*}}{|x-y|^\al} \, dx \, dy \\ \notag
\geq & \int_{\Om} \int_{\Om} \frac{\left|t u_\eps(x)\right|^{2_{\al, m}^*}\left|t u_\eps(y)\right|^{2_{\al, m}^*}}{|x-y|^\al} \, dx \, dy+22_{\al, m}^* \int_{\Om} \int_{\Om} \frac{\left|t u_\eps(x)\right|^{2_{\al, m}^*-1} v(x)\left|t u_\eps(y)\right|^{2_{\al, m}^*}}{|x-y|^\al} \, dx \, dy \\ \notag
\geq & t^{22_{\al, m}^*} \int_{\Om} \int_{\Om} \frac{\left|u_\eps(x)\right|^{2_{\al, m}^*}\left|u_\eps(y)\right|^{2_{\al, m}^*}}{|x-y|^\al} \, dx \, dy-22_{\al, m}^* t^{22_{\al, m}^*-1}|v|_{\infty} \int_{\Om} \int_{\Om} \frac{\left|u_\eps(x)\right|^{2_{\al, m}^*-1}\left|u_\eps(y)\right|^{2_{\al, m}^*}}{|x-y|^\al} \, dx \, dy \\
\geq & t^{22_{\al, m}^*} \int_{\Om} \int_{\Om} \frac{\left|u_\eps(x)\right|^{2_{\al, m}^*}\left|u_\eps(y)\right|^{2_{\al, m}^*}}{|x-y|^\al} \, dx \, dy-C_6 t^{22_{\al, m}^*-1}|v|_2 O\left(\eps^{\frac{N-2m}{2}}\right),
\end{align}
Because all norms are equivalent on ${W}_j$, being a finite-dimensional space. This implies that $t<C_6$ for some constant $C_6>0$. Taking \eqref{e23} into account, we have
\[\int_{\Om} \int_{\Om} \frac{\left|t u_{\eps}(x)\right|^{2_{\al, m}^*}\left|t u_{\eps}(y)\right|^{2_{\al, m}^*}}{|x-y|^\al} \, dx \, dy \leq 1+O\left(\eps^{\frac{N-2m}{2}}\right)|v|_2.\]
Using \eqref{e22} and \eqref{e53}, we obtain 
\begin{align*}
m_{j, \eps}&=\|u_m\|^2 -\la|u_m|_2^2\\ 
& =\int_{\Om}|D^m(v+t u_{\eps})|^2 d x-\la \int_{\Om}|v+t u_{\eps}|^2 d x \\
& \leq\left(\la^m_j-\la\right)|v|_2^2+A_{\eps}\left(\int_{\Om} \int_{\Om} \frac{|t u_{\eps}(x)|^{2_{\al, m}^*}|t u_{\eps}(y)|^{2_{\al, m}^*}}{|x-y|^\al} \, dx \, dy\right)^{\frac{N-2m}{2N-\al}}+C_4|u_{\eps}|_1|v|_2 \\
& \leq\left(\la^m_j-\la\right)|v|_2^2+A_{\eps}\left(1+|v|_2 O\left(\eps^{\frac{N-2m}{2}}\right)\right)^{\frac{N-2m}{2N-\al}}+C_4|u_{\eps}|_1|v|_2 \\
& \leq\left(\la^m_j-\la\right)|v|_2^2+A_{\eps}\left(1+|v|_2 O\left(\eps^{\frac{N-2m}{2}}\right)\right)+O\left(\eps^{\frac{N-2m }{2}}\right)|v|_2.
\end{align*}\end{enumerate}
This completes the proof.\end{proof}
\begin{lemma} \label{l12} 
Let $\la \in\left(\la_j, \la^m_{j+1}\right)$ for some $j \in \mb{N}$, then for $u \in \mathbf{P}_{j,\eps}$, we have
\[\frac{\|u\|_2^2-\la|u|_2^2}{\|u\|_{N L}^2}<S_{H, L}\]
in the following cases:
\begin{enumerate}
\item If $N\geq4m$,
\item  $N\in[2m+1, 4m)$, $\la>\la^*$.
\end{enumerate}
\end{lemma}
\begin{proof}
To give the proof, we only need to verify that
\[m_{j, \eps}=\max _{u \in \mathbf{P}_{j, \eps}, \|u\|_{NL}=1}\left(\int_{\Om}|D^m u|^2 d x-\la \int_{\Om}|u|^2 d x\right)<S_{H, L}.\]
If $t=0$ in \eqref{e21}, by the choice of $\la \in\left(\la^m_j, \la^m_{j+1}\right)$, we get that
\[m_{j,\eps} \leq\left(\la^m_j-\la\right)|v|_2^2<0<S_{H, L .}\]
Now we take $t>0$ and discuss the cases $N \geq 4m+1$, $N=4m$ and $N\in[2m+1,4m)$   separately.\\
If $N \geq 4m+1$, we have
\begin{align} \notag
m_{j, \eps} & \leq\left(\la^m_j-\la\right)|v|_2^2+\frac{\|u_{\eps}\|_2^2-\la\left|u_{\eps}\right|_2^2}{\left\|u_{\eps}\right\|_{N L}^2}\left(1+|v|_2 O\left(\eps^{\frac{N-2m}{2}}\right)\right)+O\left(\eps^{\frac{N-2m}{2}}\right)|v|_2 \\ \notag
& \leq\left(\la^m_j-\la\right)|v|_2^2+\frac{C(N, \al)^{\frac{N(N-2m)}{2m(2N-\al)}} S_{H, L}^{\frac{N}{2m}}-\la C_3 \eps^{2m}}{\left(C(N, \al)^{\frac{N}{2m}} S_{H, L}^{\frac{2N-\al}{2m}}-O\left(\eps^{N-\frac{\al}{2}}\right)\right)^{\frac{N-2m}{2N-\al}}}\left(1+|v|_2 O\left(\eps^{\frac{N-2m}{2}}\right)\right)+O\left(\eps^{\frac{N-2m}{2}}\right)|v|_2 \\ \notag
&\leq(S_{H, L}-\la C_3 \eps^{2m} + O\left(\eps^{N-2m}\right))\left(1+|v|_2 O\left(\eps^{\frac{N-2m}{2}}\right)\right)+\left(\la^m_j-\la\right)|v|_2^2+O\left(\eps^{\frac{N-2m}{2}}\right)|v|_2 \\  \notag
& \leq S_{H, L}-\la C_3 \eps^{2m}+\left(\la^m_j-\la\right)|v|_2^2+O\left(\eps^{\frac{N-2m}{2}}\right)|v|_2, \\ \notag
\end{align}
for $\eps>0$ sufficiently small. As $\la \in\left(\la^m_j, \la^m_{j+1}\right)$, we have
\begin{equation} \label{e24}
\left(\la^m_j-\la\right)|v|_2^2+O\left(\eps^{\frac{N-2m}{2}}\right)|v|_2 \leq \frac{1}{4\left(\la^m_j-\la\right)} O\left(\eps^{N-2m}\right)=O\left(\eps^{N-2m}\right).
\end{equation}
Therefore
\[m_{j, \eps} \leq S_{H, L}-\la C_3\eps^{2m}+O\left(\eps^{N-2m}\right)<S_{H, L},\]
for $\eps>0$ sufficiently small.\\

\noi If $N=4m$, equation \eqref{e24}, yield
\begin{align*}
m_{j,\eps} & \leq\left(\la^m_j-\la\right)|v|_2^2+\frac{\|u_{\eps}\|_2^2-\la\left|u_{\eps}\right|_2^2}{\left\|u_{\eps}\right\|_{N L}^2}\left(1+|v|_2 O\left(\eps^{\frac{N-2m}{2}}\right)\right)+O\left(\eps^{\frac{N-2m}{2}}\right)|v|_2 \\
& \leq\left(\la^m_j-\la\right)|v|_2^2+\frac{C(N, \al)^{\frac{N(N-2m)}{2m(2N-\al)}} S_{H, L}^{\frac{N}{2m}}+O\left(\eps^{N-2m}\right)-\la C_2 |ln{\eps}| \eps^{2m} -O\left(\eps^{2m}\right)}{\left(C(N, \al)^{\frac{N}{2m}} S_{H, L}^{\frac{2N-\al}{2m}}-O\left(\eps^{N-\frac{\al}{2}}\right)\right)^{\frac{N-2m}{2N-\al}}}\left(1+|v|_2 O\left(\eps^{\frac{N-2m}{2}}\right)\right)\\& \qquad+O\left(\eps^{\frac{N-2m}{2}}\right)|v|_2 \\
&\leq\left(S_{H, L} +O\left(\eps^{2m}\right) -\la C_2 |ln{\eps}| \eps^{2m}\right)\left(1+|v|_2 O\left(\eps^{m}\right)\right)+\left(\la^m_j-\la\right)|v|_2^2+O\left(\eps^{m}\right)|v|_2 \\
&\leq S_{H, L}-\la C_2 \eps^{2m}+ O\left(\eps^{2m}\right) \\
&<S_{H, L},
\end{align*}
for $\eps>0$ sufficiently small.\\

\noi If $2m+1\leq N< 4m$, we have
\begin{align*}
m_{j,\eps} & \leq\left(\la^m_j-\la\right)|v|_2^2+\frac{\|u_{\eps}\|_2^2-\la\left|u_{\eps}\right|_2^2}{\left\|u_{\eps}\right\|_{NL}^2}\left(1+|v|_2 O\left(\eps^{\frac{N-2m}{2}}\right)\right)+O\left(\eps^{\frac{N-2m}{2}}\right)|v|_2 \\
& \leq\left(S_{H, L}-\la C_1\eps^{N-2m} \right)\left(1+|v|_2 O\left(\eps^{\frac{N-2m}{2}}\right)\right)+\left(\la^m_j-\la\right)|v|_2^2+O\left(\eps^{\frac{N-2m}{2}}\right)|v|_2 \\
& \leq S_{H, L}+|v|_2 O\left(\eps^{\frac{N-2m}{2}}\right)-\la C_1\eps^{N-2m}-|v|_2 O\left(\eps^{\frac{3(N-2m)}{2}}\right)\\
&\leq S_{H, L}-\la C_1\eps^{N-2m}-|v|_2 O\left(\eps^{\frac{3(N-2m)}{2}}\right)\\
& <S_{H, L} ,
\end{align*}
for $\eps>0$ sufficiently small, since $\la>\la^*$ and $\la \in\left(\la^m_j, \la^m_{j+1}\right)$. The result follows.   
\end{proof}
\noi {\bf{Proof of Theorem (\ref{t4})}}\\
\begin{itemize}
\item For $N\geq4m$ and $\la >\la^m_1$. By the definition of $\mathbf{P}_{j,\eps}$ we have
\[u_m=\bar{v}+t z_{ \eps}\]
where
\[\bar{v}=v+t \sum_{i=1}^j\left(\int_{\Om} u_{\eps}e_i d x\right) e_i \in W_j\]
and
\[z_{\eps}=u_{\eps}-\sum_{i=1}^j\left(\int_{\Om} u_{\eps}e_i d x\right) e_i.\]
So, $\bar{v}$ and $z_\eps$ are orthogonal in $L^2(\Om)$. This implies that
$|u_m|_2^2=|\bar{v}|_2^2+t^2 |z_\eps|_2 ^2$.
Then
\[\mathbf{P}_{j,\eps}={W}_j \oplus \mb{R} z_{\eps}.\]
We apply Lemma \ref{l10} and also the functional $\mc{I}_\la$ satisfies the geometric structure of the linking Theorem (see \cite{6rabinowitz1986minimax}, Theorem 5.3), that is,

\begin{gather*}
\inf _{u \in \mb{E}_{j+1}, \|u\|=\rho} \mc{I}_\la(u) \geq \mu>0, \\
\sup _{u \in \mb{Y}_ j} \mc{I}_\la(u)<0
\end{gather*}
and
\[\sup _{u \in \mathbf{P}_{j,\eps},\|u\| \geq R} \mc{I}_\la(u) \leq 0,\]
where $\mu$ and $R$ are defined in Lemma \ref{l10}. We define the linking critical level of $\mc{I}_\la$, that is,
\[
c=\inf _{\gamma \in \Gamma} \max _{u \in V} \mc{I}_\la(\gamma(u))>0,
\]
where
\[\Gamma:=\left\{\gamma \in C\left(\overline{V}, H_0^m(\Om)\right): \gamma=i d \text { on } \partial V\right\}\]
and
\[V:=\left(\overline{B_R} \cap {W}_{\jmath}\right) \oplus\left\{r z_\eps: r \in(0, R)\right\}.\]
For any $\gamma \in \Gamma$, we have
\[c\leq \max _{u \in V} \mc{I}_\la(\gamma(u)).\]
In particular, if we take $\gamma=i d$ on $\overline{V}$, then
\[c \leq \max _{u \in V} \mc{I}_\la(u) \leq \max _{u \in \mathbf{P}_{j,\eps}} \mc{I}_\la(u).\]
Also for any $u \in \mb{H}_0^m(\Om) \backslash\{0\}$,
\[\max _{t \geq 0} \mc{I}_\la(t u)=\frac{N+2m-\al}{4 N-2\al}\left(\frac{\|u\|_2^2-\la|u|_2^2}{\|u\|_{N L}^2}\right)^{\frac{2N-\al}{N+2m-\al}}.\]
Since $\mathbf{P}_{j,\eps}$ is a linear space, therefore we have
\[\max _{u \in \mathbf{P}_{j,\eps}} \mc{I}_\la(u)=\max _{u \in \mathbf{P}_{j,\eps}, t \neq 0} \mc{I}_\la\left(|t| \frac{u}{|t|}\right)=\max _{u \in \mathbf{P}_{j,\eps}, t>0} \mc{I}_\la(t u) \leq \max _{u \in \mathbf{P}_{j,\eps}, t \geq 0} \mc{I}_\la(t u).\]
Thus, by Lemma \ref{l12} we have
\begin{align*}
c & \leq \max _{u \in \mathbf{P}_{j,\eps},t \geq 0} \mc{I}_\la(tu) \\
& =\max _{u \in \mathbf{P}_{j,\eps}} \frac{N+2m-\al}{4 N-2\al}\left(\frac{\|u\|_2^2-\la|u|_2^2}{\|u\|_{N L}^2}\right)^{\frac{2N-\al}{N+2m-\al}} \\
& <\frac{N+2m-\al}{4N-2\al} S_{H, L}^{\frac{2N-\al}{N+2m-\al}}.
\end{align*}
Therefore, by  the linking Theorem and Lemma \ref{l4}, we conclude that problem \eqref{e5} has a nontrivial solution $u \in \mb{H}_0^m(\Om)$ with critical value $c\geq \mu$.\\

\item For $N\in[2m+1,4m)$ and $\la>\la^*$. We consider the two cases: 
\begin{enumerate}
\item $\la^m_1>\la^*$
\item $\la^m_1\leq\la^*$
\end{enumerate} 
{\bf{Case 1. $\la^m_1>\la^*$.}}\\
For this case, we will use the Mountain Pass Theorem if $\la \in\left(\la^*, \la^m_1\right)$ while the Linking Theorem if $\la \in\left(\la^m_j, \la^m_{j+1}\right)$ for some $j \in \mb{N}$.\\
If $\la \in (\la^*, \la^m_1)$, by Lemma \ref{l19} and the mountain pass theorem without $(PS)$ condition \cite{4willem1996minimax}, there exists a (PS) sequence $\{u_n\}$ such that $\mc{I}_\la(u_n) \to c$ and $\mc{I}_\la^{\prime}(u_n) \to  0$ in $H_0^m(\Om)^{-1}$ at the Mountain Pass level $c$. From Lemma \ref{l18}, we have there exists $v \in \mb{H}_0^m(\Om) \backslash\{0\}$ such that
\[\frac{\|v\|_2^2-\la|v|_2^2}{\|v\|_{N L}^2}<S_{H, L}.\]
Thus,
\begin{align*}
0<\max _{t \geq 0} \mc{I}_\la(t v) & =\max _{t \geq 0}\left\{\frac{t^2}{2} \int_{\Om}|D^m v|^2 d x-\frac{\la t^2}{2} \int_{\Om} |v|^2 d x-\frac{t^{22_{\al, m}^*}}{22_{\al, m}^*} \int_{\Om} \int_{\Om} \frac{|v(x)|^{2_{\al, m}^*}|v(y)|^{2_{\al, m}^*}}{|x-y|^\al} \, dx \, dy\right\} \\
& <\frac{N+2m-\al}{4N-2\al}S_{H, L}^{\frac{2N-\al}{N+2m-\al}}.
\end{align*}
By the definition of $c$, we know $c<\frac{N+2m-\al}{4N-2\al}S_{H, L}^{\frac{2N-\al}{N+2m-\al}}$.\\
From Lemma \ref{l4}, we obtain $\{u_n\}$ contains a convergent subsequence. So, we have $\mc{I}_\la$ has a critical value $c\in\left(0,\frac{N+2m-\al}{4N-2\al}S_{H, L}^{\frac{2N-\al}{N+2m-\al}}\right)$ and problem \eqref{e5} has a nontrivial solution.\\
If $\la \in (\la^m_j, \la^m_{j+1})$ for some $j \in \mb{N}$, the proof follows similarly as in the above case $N\geq4m$ and $\la>\la^m_1$.
Therefore, the Linking Theorem and Lemma \ref{l4}
yield that problem \eqref{e5} admits a solution $u \in \mb{H}_0^m(\Om)$ with critical value $c \geq \mu$. Since $\mu>0=\mc{I}_\la(0)$, we deduce that $u$ is not identically zero.\\
{\bf{Case $2. \quad\la^m_1 \leq\la^*$}}.\\
In this case, we only consider $\la \in (\la^m_j, \la^m_{j+1})$ for some $j \in \mb{N}$ and $\la>\la^*$. We can argue as in the last part of Case 1. In this way, we get that for any $\la>\la^*$ different from an eigenvalue of $(-\Delta)^m$, problem \eqref{e5} admits a solution $u \in \mb{H}_0^m(\Om)$ with critical value $c\geq \mu$ and $u$ is not identically zero.\\
\end{itemize}
\section{\bf{Perturbation with a superlinear nonlocal term}}
\noi In the last section, we show the existence of a solution to the following problem,
\[
\left\{\begin{aligned}
(-\Delta)^m u&=\left(\int_{\Om} \frac{|u|^{2_{\al, m}^*}}{|x-y|^\al} d y\right)|u|^{2_{\al, m}^*-2} u+\la\left(\int_{\Om} \frac{|u|^q}{|x-y|^\al} d y\right)|u|^{q-2} u \,&& \text { in } \Om, \\
\nabla^k u&=0\ \forall \, \,  |k|\leq m-1 \,&& \text{ on}\ \partial \Om  .\end{aligned}\right.
\]
We know that the domain of this problem is bounded, so the Sobolev embedding and the Hardy-Littlewood-Sobolev inequality imply that the integral
$\int_{\Om} \int_{\Om} \frac{|u(x)|^q|u(y)|^q}{|x-y|^\al} \, dx \, dy,$
is well defined if $\frac{2N-\al}{2N} \leq q < \frac{2N-\al}{N-2m}=2_{\al, m}^*.$\\
Now, we introduce the energy functional
\[
\mc{I_\la}(u)=\frac{1}{2} \int_{\Om}|D^m u|^2 d x-\frac{1}{22_{\al, m}^*} \int_{\Om} \int_{\Om} \frac{|u(x)|^{2_{\al, m}^*}|u(y)|^{2_{\al, m}^*}}{|x-y|^\al} \, dx \, dy-\frac{\la}{2q} \int_{\Om} \int_{\Om} \frac{|u(x)|^q|u(y)|^q}{|x-y|^\al} \, dx \, dy,
\]
we can observe that the given energy functional belongs to $C^1(H_0^m(\Om), \mb{R})$. Moreover, $u$ is a weak solution of \eqref{e6} if and only if $u$ is a critical point of functional $\mc{I_\la}$.\\
We first show that $\mc{I_\la}$ satisfies mountain pass geometry.
\begin{lemma}\label{l6}
If $1<q<2_{\al, m}^*$ and $\la>0$, then, the functional $\mc{I_\la}$ satisfies the following properties:
\begin{enumerate}
\item There exist $\mu, \rho>0$ such that $\mc{I_\la}(u) \geq \mu$ for $\|u\|=\rho$.
\item There exists $e \in H_0^m(\Om)$ with $\|e\|>\rho$ such that $\mc{I_\la}(e)<0$.
\end{enumerate}
\end{lemma}
\begin{proof}
The proof follows in a similar manner as in Lemma \ref{l1}.   
\end{proof}
\begin{lemma}\label{l7}
If $1<q<2_{\al, m}^*$ and $\la>0$ and $\{u_n\}$ is a $(P S)_c$ sequence of $\mc{I_\la}$, then $\{u_n\}$ is bounded. Moreover, if there exist $u_0\in H_0^m(\Om)$ be the weak limit of $\{u_n\}$, then $u_0$ is a weak solution of problem \eqref{e6} and $\mc{I_\la}(u_0)\geq0$.
\end{lemma}
\begin{proof}
The proof follows the same as in Lemma \ref{l2}.\\

\end{proof}
\begin{lemma}\label{l8}
If $1<q<2_{\al, m}^*$, $\la>0$ and $\{u_n\}$ is a $(P S)_c$ sequence of $\mc{I}_\la$ with
\[c<\frac{N+2m-\al}{4 N-2 \al} S_{H, L}^{\frac{2N-\al}{N+2m-\al}},\]
then $\{u_n\}$ has a convergent subsequence.
\end{lemma}
\begin{proof}
As we have already obtained that $u_0$ is the weak limit of $\{u_n\}$ in Lemma \ref{l7}.\\ We define $v_n:=u_n-u_0$, then we can easily see that $v_n \weak 0$ in $H_0^m(\Om)$ and $v_n \rightarrow 0$ a.e. in $\Om$. Moreover, by Lemma \ref{l3}, we know
\[\DD\int_{\Om} \int_{\Om} \frac{|u_n(x)|^q |u_n(y)|^q}{|x-y|^\al} \, dx \, dy=\int_{\Om} \int_{\Om} \frac{\left|v_n(x)\right|^q\left|v_n(y)\right|^q}{|x-y|^\al} \, dx \, dy+\int_{\Om} \int_{\Om} \frac{|u_0(x)|^q\left|u_0(y)\right|^q}{|x-y|^\al} \, dx \, dy+o_n(1).\]
Similar to the proof of Lemma \ref{l4}, we have
\[
c \leftarrow \mc{I_\la}\left(u_n\right) \geq \frac{1}{2} \int_{\Om}\left|D^mv_n\right|^2 d x-\frac{1}{22_{\al, m}^*} \int_{\Om} \int_{\Om} \frac{\left|v_n(x)\right|^{2_{\al, m}^*}\left|v_n(y)\right|^{2_{\al, m}^*}}{|x-y|^\al} \, dx \, dy+o_n(1),
\]
since $\mc{I_\la}(u_0) \geq 0$ and
\[
\int_{\Om} \int_{\Om} \frac{\left|v_n(x)\right|^q\left|v_n(y)\right|^q}{|x-y|^\al} \, dx \, dy \leq C(N, \al)\left|v_n^q\right|_{\frac{2N}{2N-\al}}^2 \rightarrow 0,
\]
as $n \rightarrow+\infty$. Reiterating the same arguments in the proof of Lemma \ref{l4}, we have
\[\left\|u_n-u_0\right\| \rightarrow 0,\]
as $n \rightarrow+\infty$. Hence, we have the required existence of a convergent subsequence.
\end{proof}
\begin{lemma} \label{l9}
Let $1<q<2_{\al,m}^*$ and $u_{\eps}$ as defined in \eqref{e10}. If one of the following conditions holds:
\begin{enumerate}
\item $N>\frac{2m(q+1)-\al}{q-1}$ and $\la>0$,
\item $N \leq \frac{2m(q+1)-\al}{q-1}$ and $\la$ is sufficiently large,
then there exists $\eps$ such that 
\end{enumerate}
\[\sup _{t \geq 0} \mc{I_\la}\left(t u_{\eps}\right)<\frac{N+2m-\al}{4 N-2\al} S_{H, L}^{\frac{2N-\al}{N+2m-\al}}.\]
\end{lemma}
\begin{proof}\begin{enumerate}
\item $N>\frac{2m(q+1)-\al}{q-1}$.\\
In this case, we will estimate the convolution part; we know
 \begin{align}\label{e14}
\int_{\Om} \int_{\Om} \frac{\left|u_{\eps}(x)\right|^q\left|u_{\eps}(y)\right|^q}{|x-y|^\al} \, dx \, dy \geq & \int_{B_r} \int_{B_r} \frac{\left|u_{\eps}(x)\right|^q\left|u_{\eps}(y)\right|^q}{|x-y|^\al} \, dx \, dy \notag\\
= & \int_{B_r} \int_{B_r} \frac{\left|U_{\eps}(x)\right|^q\left|U_{\eps}(y)\right|^q}{|x-y|^\al} \, dx \, dy \notag\\
= & \int_{\Om} \int_{\Om} \frac{\left|U_{\eps}(x)\right|^q\left|U_{\eps}(y)\right|^q}{|x-y|^\al} \, dx \, dy-2 \int_{\Om \backslash B_r}\int_{B_r} \frac{\left|U_{\eps}(x)\right|^q\left|U_{\eps}(y)\right|^q}{|x-y|^\al} \, dx \, dy \notag\\
& \quad-\int_{\Om \backslash B_r}\int_{\Om \backslash B_r} \frac{\left|U_{\eps}(x)\right|^q\left|U_{\eps}(y)\right|^q}{|x-y|^\al} \, dx \, dy \\
:= & \mathrm{B_1}-2 \mathrm{B_2}-\mathrm{B_3},\notag
\end{align}
where
\[\mathrm{B_1}:=\int_{\Om} \int_{\Om} \frac{\left|U_{\eps}(x)\right|^q\left|U_{\eps}(y)\right|^q}{|x-y|^\al} \, dx \, dy, \quad \mathrm{B_2}:=\int_{\Om \backslash B_r} \int_{B_r}\frac{\left|U_{\eps}(x)\right|^q\left|U_{\eps}(y)\right|^q}{|x-y|^\al}\, dx \, dy\]
and
\[\mathrm{B_3}:=\int_{\Om \backslash B_r}\int_{\Om \backslash B_r} \frac{\left|U_{\eps}(x)\right|^q\left|U_{\eps}(y)\right|^q}{|x-y|^\al} \, dx \, dy .\]
We are going to estimate $\mathrm{B_1}, \mathrm{B_2}$ and $\mathrm{B_3}$. By direct computation, we know, for $\eps<1$,

\begin{align*}
\mathrm{B_1} & =\eps^{-(N-2m) q}C_{N,m}^{\frac{(N-2m)2q}{4m}} \int_{\Om} \int_{\Om} \frac{1}{\left(1+\left|\frac{x}{\eps}\right|^2\right)^{\frac{(N-2m) q}{2}}|x-y|^\al\left(1+\left|\frac{y}{\eps}\right|^2\right)^{\frac{(N-2m) q}{2}}} \, dx \, dy \\
& \geq \eps^{-(N-2m) q}C_{N,m}^{\frac{(N-2m)2q}{4m}} \int_{B_r} \int_{B_r} \frac{1}{\left(1+\left|\frac{x}{\eps}\right|^2\right)^{\frac{(N-2m) q}{2}}|x-y|^\al\left(1+\left|\frac{y}{\eps}\right|^2\right)^{\frac{(N-2) q}{2}}} \, dx \, dy\\
& =\eps^{-(N-2m) q}C_{N,m}^{\frac{(N-2m)2q}{4m}}\eps^{2N-\al} \int_{B_\frac{r}{\eps}} \int_{B_\frac{r}{\eps}} \frac{1}{\left(1+|x|^2\right)^{\frac{(N-2m) q}{2}}|x-y|^\al\left(1+|y|^2\right)^{\frac{(N-2m) q}{2}}} \, dx \, dy \\
& = O\left(\eps^{2N-\al-(N-2m) q}\right),
\end{align*}
i.e
\begin{equation*}
\mathrm{B_1}\geq O\left(\eps^{2N-\al-(N-2m) q}\right).
\end{equation*}
\begin{align*}
\mathrm{B_2} &=\eps^{(N-2m) q}C_{N,m}^{\frac{(N-2m)2q}{4m}}
\int_{\Om \setminus B_r} \int_{B_r}
\frac{1}{\left(\eps^2+|x|^2\right)^{\frac{(N-2m)q}{2}}
|x-y|^\al
\left(\eps^2+|y|^2\right)^{\frac{(N-2m)q}{2}}}\, dx\, dy \notag\\
& \leq O\left(\eps^{(N-2m)q}\right)\left(\int_{\Om \setminus B_r} \frac{1}{\left(\eps^2+|x|^2\right)^{\frac{(N-2m)qN}{2N-\al}}} d x\right)^{\frac{2N-\al}{2N}}\left(\int_{B_r} \frac{1}{\left(\eps^2+|y|^2\right)^{\frac{
(N-2m)qN}{2N-\al}}} d y\right)^{\frac{2N-\al}{2N}}\notag \\
& = O\left(\eps^{(N-2m)q}\right)\left(\int_{\Om \backslash B_r} \frac{1}{\left(|x|^2\right)^{\frac{(N-2m)qN}{2N-\al}}} d x\right)^{\frac{2N-\al}{2N}}\left(\int_{B_r} \frac{1}{\left(\eps^2+|y|^2\right)^{\frac{
(N-2m)qN}{2N-\al}}} d y\right)^{\frac{2N-\al}{2N}} \notag \\
& \leq O\left(\eps^{(N-2m)q}\right)\left(\int_{B_r} \frac{1}{\left(\eps^2+|y|^2\right)^{\frac{
(N-2m)qN}{2N-\al}}} d y\right)^{\frac{2N-\al}{2N}} \notag\\
& = O\left(\eps^{\frac{2N-\al}{2}}\right)\left(\int_{0}^{\frac{r}{\eps}} \frac{t^{N-1}}{\left(1+|t|^2\right)^{\frac{
(N-2m)qN}{2N-\al}}} dt\right)^{\frac{2N-\al}{2N}}\notag\\
& \leq O\left(\eps^{\frac{2N-\al}{2}}\right)\left(\int_{0}^{+\infty} \frac{t^{N-1}}{\left(1+|t|^2\right)^{\frac{
(N-2m)qN}{2N-\al}}} dt\right)^{\frac{2N-\al}{2N}}.
\end{align*}
Since $N>\frac{2m(q+1)-\al}{q-1}>2m+\frac{4m-\al}{2(q-1)}$ if $\al<4m$ and $N>2m+\frac{4m-\al}{2(q-1)}$ if $\al \geq 4m$, we know $\frac{2(N-2m) q N}{2N-\al}>N$,\\ therefore
\begin{equation*}
\mathrm{B_2} \leq O\left(\eps^{\frac{2N-\al}{2}}\right).
\end{equation*}
\begin{align}\label{e18}
\mathrm{B_3} &= \eps^{(N-2m) q} C_{N,m}^{\frac{(N-2m)2q}{4m}}
\int_{\Om \setminus B_r} \int_{\Om \setminus B_r}
\frac{1}{\left(\eps^2 + |x|^2\right)^{\frac{(N-2m) q}{2}}|x-y|^\al\left(\eps^2 + |y|^2\right)^{\frac{(N-2m) q}{2}}}\, dx \, dy\notag \\
&\leq \eps^{(N-2m) q} C_{N,m}^{\frac{(N-2m)2q}{4m}}
\int_{\Om \setminus B_r} \int_{\Om \setminus B_r}
\frac{1}{|x|^{(N-2m) q}|x-y|^\al|y|^{(N-2m) q}}\, dx \, dy\notag \\
&\leq O\!\left( \eps^{(N-2m) q} \right).
\end{align}
It follows from \eqref{e14}- \eqref{e18} that

\begin{align}\label{e19}
\int_{\Om} \int_{\Om} \frac{\left|u_{\eps}(x)\right|^q\left|u_{\eps}(y)\right|^q}{|x-y|^\al} \, dx \, dy & \geq O\left(\eps^{2N-\al-(N-2m) q}\right)-O\left(\eps^{\frac{2N-\al}{2}}\right)-O\left(\eps^{(N-2m) q}\right)\notag \\
& =O\left(\eps^{2N-\al-(N-2m) q}\right)-O\left(\eps^{\min \left\{\frac{2N-\al}{2},   (N-2m) q\right\}}\right).
\end{align}
By \eqref{e11}, \eqref{e12} and \eqref{e19}, we have
\begin{align*}
\mc{I_\la}\left(t u_{\eps}\right)=&\frac{t^2}{2} \int_{\Om}\left|D^m u_{\eps}\right|^2 d x-\frac{t^{22_{\al, m}^*}}{22_{\al, m}^*} \int_{\Om} \int_{\Om} \frac{\left|u_{\eps}(x)\right|^{2_{\al, m}^*}\left|u_{\eps}(y)\right|^{2_{\al, m}^*}}{|x-y|^\al} \, dx \, dy-\frac{\la t^{2 q}}{2 q} \int_{\Om} \int_{\Om} \frac{\left|u_{\eps}(x)\right|^q\left|u_{\eps}(y)\right|^q}{|x-y|^\al} \, dx \, dy\\
\leq&  \frac{t^2}{2}\left(C(N, \al)^{\frac{(N-2m)}{(2N-\al)}\frac{N}{2m}} S_{H, L}^{\frac{N}{2m}}+O\left(\eps^{N-2m}\right)\right)-\frac{t^{22_{\al, m}^*}}{22_{\al, m}^*}\left(C(N, \al)^{\frac{N}{2m}} S_{H, L}^{\frac{2N-\al}{2m}}-O\left(\eps^{N-\frac{\al}{2}}\right)\right) \\
\qquad \, & -\frac{t^{2q}}{2q}\left(O\left(\eps^{2N-\al-(N-2m)q}\right)-O\left(\eps^{\min \left\{\frac{2N-\al}{2}, (N-2m)q\right\}}\right) \right)\\
&:=  K(t).
\end{align*}
It is clear that $K(t) \rightarrow-\infty$ as $t \rightarrow+\infty$, and for a very small neighbourhood of zero, $K(t)$ becomes positive. It gives the existence of $t_{\eps}>0$ such that $\displaystyle\sup_{t>0} K(t)$ is attained at $t_{\eps}$.\\
Therefore $\DD\frac{d}{dt} K(t_\eps) =0$,
\begin{align*}
t_{\eps}\left(C(N, \al)^{\frac{(N-2m)}{(2N-\al)} \frac{N}{2m}} S_{H, L}^{\frac{N}{2m}}+O\left(\eps^{N-2m}\right)\right) & -t_{\eps}^{22_{\al, m}^*-1}\left(C(N, \al)^{\frac{N}{2m}} S_{H, L}^{\frac{2N-\al}{2m}}-O\left(\eps^{N-\frac{\al}{2}}\right)\right) \\
& -t^{2 q-1}\left(O\left(\eps^{2N-\al-(N-2m) q}\right)-O\left(\eps^{\min \left\{\frac{2N-\al}{2}, (N-2m)q\right\}}\right)\right)=0,
\end{align*}
since $\left(N>\frac{2m(q+1)-\al}{q-1}\right)$ and $N \geq 2m+1$, we know
\begin{equation}\label{e20}
2N-\al-(N-2m)q<\min \left\{\frac{2N-\al}{2}, (N-2m) q\right\},
\end{equation}
which implies that,
\[
O\left(\eps^{2N -\al-(N-2m) q}\right)-O\left(\eps^{\min \left\{\frac{2N-\al}{2}, (N-2m) q\right\}}\right) \geq 0,
\]
if $\eps$ is small enough. And so
\[
t_{\eps}<\left(\frac{C(N, \al)^{\frac{(N-2m)}{(2N-\al)}\frac{N}{2m}} S_{H, L}^{\frac{N}{2m}}+O\left(\eps^{N-2m}\right)}{C(N, \al)^{\frac{N}{2m}} S_{H, L}^{\frac{2N-\al}{2m}}-O\left(\eps^{N-\frac{\al}{2}}\right)}\right)^{\frac{1}{22_{\al, m}^*-2}}:=\mc{T}_{\eps},\]
and there exists $t_0>0$ such that for $\eps>0$ small enough
\[t_{\eps}>t_0.\]
Notice that the function
\[
t \mapsto \frac{t^2}{2}\left(C(N, \al)^{\frac{(N-2m)}{(2N-\al)}\frac{N}{2m}} S_{H, L}^{\frac{N}{2m}}+O\left(\eps^{N-2m}\right)\right)-\frac{t^{22_{\al, m}^*}}{22_{\al, m}^*}\left(C(N, \al)^{\frac{N}{2m}} S_{H, L}^{\frac{2N-\al}{2m}}-O\left(\eps^{N-\frac{\al}{2}}\right)\right)
\]
is increasing on $\left[0, \mc{T}_\eps\right]$, Using $t_0<t_{\eps}<\mc{T}_\eps$ and \eqref{e19}, we have
\begin{align*}
\max_{t \geq 0} \mc{I}_\la(t u_{\eps}) 
\leq & \frac{N+2m-\al}{4N-2\al}
\left(\frac{C(N, \al)^{\frac{N(N-2m)}{2m(2N-\al)}} \, S_{H, L}^{\frac{N}{2m}} 
+ O\left(\eps^{N-2m}\right)}{\left(C(N, \al)^{\frac{N}{2m}} S_{H, L}^{\frac{2N-\al}{2m}} 
- O\left(\eps^{N-\frac{\al}{2}}\right)
\right)^{\frac{N-2m}{2N-\al}}}\right)^{\frac{2N-\al}{N+2m-\al}}- O\left(\eps^{2N-\al-(N-2m) q}\right)\\
&\qquad \, \, + O\left(\eps^{\min \left\{ \frac{2N-\al}{2}, (N-2m) q \right\}}\right)
\\
\leq & \frac{N+2m-\al}{4N-2\al} \,S_{H, L}^{\frac{2N-\al}{N+2m-\al}}
+ O\left(\eps^{\min \left\{\frac{2N-\al}{2}, N-2m\right\}}\right)
- O\left(\eps^{2N-\al-(N-2m) q}\right)
+ O\left(\eps^{\min \left\{\frac{2N-\al}{2}, (N-2m) q\right\}}\right).
\end{align*}
The assumptions $N>\frac{2m(q+1)-\al}{q-1}$ and $1<q<2_{\al, m}^*$ together with \eqref{e20} imply that
\[
\max _{t \geq 0} \mc{I_\la}\left(t u_{\eps}\right)<\frac{N+2m-\al}{4 N-2 \al} S_{H, L}^{\frac{2N-\al}{N+2m-\al}},
\]
if $\eps$ is small enough.
\item $N \leq \frac{2m(q+1)-\al}{q-1}$. For any fixed $\eps$ in \eqref{e10}, assume that $\displaystyle\max_{t \geq 0} \mc{I_\la}(t u_{\eps})$ is obtained at some $t_\la>0$. Reiterating the same arguments in the proof of Lemma \ref{l5}, we know $t_\la \rightarrow 0$ as $\la \rightarrow+\infty$ and $\displaystyle\max _{t \geq 0} \mc{I_\la}\left(t u_{\eps}\right) \rightarrow 0$, as $\la \rightarrow+\infty$, which gives us the conclusion for the case $N \leq \frac{2m(q+1)-\al}{q-1}$.
\end{enumerate}
This completes the proof.
\end{proof}
\noi{\bf{Proof of Theorem \ref{t2}}}
By Lemma \ref{l6}, Lemma \ref{l9} and the Mountain Pass Theorem without $(PS)$ condition (\cite{4willem1996minimax}), there exists a $(PS)_c$ sequence $\{u_n\}$ of $\mc{I_\la}$ with $c<\frac{N+2m-\al}{4 N-2\al} S_{H, L}^{\frac{2N-\al}{N+2m-\al}}$ if one of the following conditions holds:
\begin{enumerate}   
\item $N>\frac{2m(q+1)-\al}{q-1}$ and $\la>0$,
\item $N \leq \frac{2m(q+1)-\al}{q-1}$ and $\la$ is sufficiently large. 
\end{enumerate} 
Applying Lemma \ref{l8}, we know that $\{u_n\}$ contains a convergent subsequence. So we have $\mc{I_\la}$ has a critical value $c \in\left(0, \frac{N+2m-\al}{4 N-2\al} S_{H, L}^{\frac{2N-\al}{N+2m-\al}}\right)$ and problem \eqref{e6} has a nontrivial solution.\qed\\
{\bf{ Proof of Theorem \ref{t3}:}}  Proof of Theorem \ref{t3} is similar to that of Theorem \ref{t2} and the main difference is that the $(PS)_c$ condition holds below the critical level $\frac{m}{N} S^{\frac{N}{2m}}$. From Lemma 4.3 of \cite{22shang2014multiple}, we know that
\[\int_{\Om}\left|D^m u_{\eps}\right|^2 d x=S^{\frac{N}{2m}}+O\left(\eps^{N-2m}\right)\]
and
\[\int_{\Om}\left|u_{\eps}\right|^{2^*_m} d x=S^{\frac{N}{2m}}+O\left(\eps^N\right).\]
For  part $(1)$ of Lemma \ref{l9}, we have

\begin{align*}
\max _{t \geq 0} \mc{I_\la}\left(t u_{\eps}\right) & \leq \frac{m}{N}\left(\frac{S^{\frac{N}{2m}}+O\left(\eps^{N-2m}\right)}{\left(S^{\frac{N}{2m}}+O\left(\eps^N\right)\right)^{\frac{2}{2_m^*}}}\right)^{\frac{N}{2m}}-O\left(\eps^{2N-\al-(N-2m) q}\right)+O\left(\eps^{\min \left\{\frac{2N-\al}{2m}, (N-2m) q\right\}}\right) \\
& <\frac{m}{N} S^{\frac{N}{2m}}+O\left(\eps^{N-2m}\right)-O\left(\eps^{2N-\al-(N-2m) q}\right)+O\left(\eps^{\min \left\{\frac{2N-\al}{2m}, (N-2m)q\right\}}\right) \\
& <\frac{m}{N} S^{\frac{N}{2m}},
\end{align*}
thanks to $N>\frac{2m(q+1)-\al}{q-1}$.\\  Now the rest of the proof follows same as in Theorem \ref{t2}.\qed\\

\noi{\bf Acknowledgment:} 
The second author would like to thank the Anusandhan National Research Foundation,  Government of India, for the financial support under the grant ANRF/ARGM/2025/001486/MTR. 
\bibliographystyle{plain}
\bibliography{references.bib}

@article{1gao2017nonlocal,
  title={On nonlocal {C}hoquard equations with {H}ardy--{L}ittlewood--{S}obolev critical exponents},
  author={Gao, F. and Yang, M.},
  journal={Journal of Mathematical Analysis and Applications},
  volume={448},
  number={2},
  pages={1006--1041},
  year={2017},
  publisher={Elsevier}
}

@article{2ackermann2004periodic,
  title={On a Periodic {S}chr{\"o}dinger Equation with Nonlocal Superlinear part},
  author={Ackermann, N.},
  journal={Mathematische Zeitschrift},
  volume={248},
  number={2},
  pages={423--443},
  year={2004},
  publisher={Springer}
}

@article{3brezis1983relation,
  title={A Relation between Pointwise Convergence of Functions and Convergence of Functionals},
  author={Br{\'e}zis, H. and Lieb, E.},
  journal={Proceedings of the American Mathematical Society},
  volume={88},
  number={3},
  pages={486--490},
  year={1983}
}

@article{4willem1996minimax,
  title={Minimax theorems. {P}rogress Nonlinear Differential Equations Appl. 24},
  author={Willem, M.},
  journal={Birkh{\"a}user Boston, Inc., Boston, MA},
  volume={2},
  pages={5},
  year={1996}
}

@book{6rabinowitz1986minimax,
  author={Rabinowitz, P. H.},
  title={Minimax methods in critical point theory with applications to differential equations},
  number={65},
  year={1986},
  publisher={American Mathematical Soc.}
}

@article{1310.57262/ade/1366030750,
author = {O. Lakkis},
title = {{Existence of solutions for a class of semilinear polyharmonic equations with critical exponential growth}},
volume = {4},
journal = {Advances in Differential Equations},
number = {6},
publisher = {Khayyam Publishing, Inc.},
pages = {877 -- 906},
year = {1999},
doi = {10.57262/ade/1366030750},
URL = {https://doi.org/10.57262/ade/1366030750}
}

@article{14article,
author = {Lam, N. and Lu, G.},
title = {Existence of nontrivial solutions to polyharmonic equations with subcritical and critical exponential growth},
volume = {32},
journal = {Discrete and Continuous Dynamical Systems},
  volume={32},
  number={6},
  pages={2187--2205},
  year={2012},
doi = {10.3934/dcds.2012.32.2187}
}

@article{22shang2014multiple,
  title={Multiple nontrivial solutions for a class of semilinear Polyharmonic equations},
  author={Shang, Y. and Li, W.},
  journal={Acta Mathematica Scientia},
  volume={34},
  number={5},
  pages={1495--1509},
  year={2014},
  publisher={Elsevier}
}

@article{23ferrero2009solutions,
  title={On solutions of second and fourth order elliptic equations with power-type nonlinearities},
  author={Ferrero, A. and Warnault, G.},
  journal={Nonlinear Analysis: Theory, Methods and Applications},
  volume={70},
  number={8},
  pages={2889--2902},
  year={2009},
  publisher={Elsevier}
}

@article{24myers1998thin,
  title={Thin films with high surface tension},
  author={Myers, T. G.},
  journal={SIAM review},
  volume={40},
  number={3},
  pages={441--462},
  year={1998},
  publisher={SIAM}
}

@article{25grunau1995positive,
  title={Positive solutions to semilinear polyharmonic {D}irichlet problems involving critical {S}obolev exponents},
  author={Grunau, H. C.},
  journal={Calculus of Variations and Partial Differential Equations},
  volume={3},
  number={2},
  pages={243--252},
  year={1995},
  publisher={Springer}
}

@article{26swanson1992best,
  title={The best {S}obolev constant},
  author={Swanson, C. A. },
  journal={Applicable Analysis},
  volume={47},
  number={1-4},
  pages={227--239},
  year={1992},
  publisher={Taylor \& Francis}
}

@article{28riesz1949integrale,
title={L'int{\'e}grale de Riemann-Liouville et le probl{\`e}me de Cauchy},
  author={Riesz, M.},
  journal={Acta. Math},
  volume={81},
  pages={10--16},
  year={1949}
}

@article{29edmunds1990critical,
  title={Critical exponents, critical dimensions and the biharmonic operator},
  author={Edmunds, D. E. and Fortunato, D. and Jannelli, E.},
  journal={Archive for Rational Mechanics and Analysis},
  volume={112},
  number={3},
  pages={269--289},
  year={1990},
  publisher={Springer-Verlag Berlin/Heidelberg}
}

@article{31colasuonno2012multiple,
  title={Multiple solutions for an eigenvalue problem involving $p$-{L}aplacian type operators},
  author={Colasuonno, F. and Pucci, P. and Varga, C.},
  journal={Nonlinear Analysis: Theory, Methods and Applications},
  volume={75},
  number={12},
  pages={4496--4512},
  year={2012},
  publisher={Elsevier}
}

@article{32pucci1990critical,
  title={Critical exponents and critical dimensions for polyharmonic operators},
  author={Pucci, P. and Serrin, J.},
  journal={Journal de Math{\'e}matiques Pures et Appliqu{\'e}es},
  volume={69},
  number={1},
  pages={55--83},
  year={1990},
  publisher={-PARIS, FRANCE: EDITIONS ELSEVIER-Paris FRANCE: Bachelier 1836--PARIS FRANCE~…}
}

@article{33gazzola1998critical,
  title={Critical growth problems for polyharmonic operators},
  author={Gazzola, F.},
  journal={Proceedings of the Royal Society of Edinburgh Section A: Mathematics},
  volume={128},
  number={2},
  pages={251--263},
  year={1998},
  publisher={Royal Society of Edinburgh Scotland Foundation}
}

@article{35grunau1996conjecture,
  title={On a conjecture of {P}. {P}ucci and {J}. {S}errin},
  author={Grunau, H. C.},
  journal={Analysis},
  volume={16},
  number = {4},
  pages={399--403},
  year={1996}
}

@article{36pucci1986general,
  title={A general variational identity},
  author={Pucci, P. and Serrin, J.},
  journal={Indiana University mathematics journal},
  volume={35},
  number={3},
  pages={681--703},
  year={1986},
  publisher={JSTOR}
}

@article{37article,
author = {Gao, F. and Yang, M.},
title = {On the {B}r{\'e}zis-{N}irenberg type critical problem for nonlinear {C}hoquard equation},
journal = {Science China Mathematics},
volume = {61},
number={7},
pages = {1219-1242},
year = {2018},
doi = {10.1007/s11425-016-9067-5}
}

@article{38chen2025new,
  title={New type of solutions for the critical polyharmonic equation},
  author={Chen, W. and Wang, Z.},
  journal={Journal of Differential Equations},
  volume={429},
  pages={678--715},
  year={2025},
  publisher={Elsevier}
}

@article{39ge2011critical,
  title={A critical elliptic problem for polyharmonic operators},
  author={Ge, Y. and Wei, J. and Zhou, F.},
  journal={Journal of Functional Analysis},
  volume={260},
  number={8},
  pages={2247--2282},
  year={2011},
  publisher={Elsevier}
}

@article{40dou2012solutions,
  title={Solutions for polyharmonic elliptic problems with critical nonlinear in symmetric domains.},
  author={Dou, J. and Guo, Q. and Wang, Z. Q.},
  journal={Communications on Pure and Applied Analysis},
  volume={11},
  number={2},
  pages = {453-464},
  year={2012}
}

@article{41lei2025multiple,
  title={Multiple solutions for a bibarmonic elliptic equation with critical {C}hoquard type nonlinearity},
  author={Lei, C. and Lei, J. and Suo, H.},
  journal={Quaestiones Mathematicae},
  volume={48},
  number={1},
  pages={121--141},
  year={2025},
  publisher={Taylor \& Francis}
}

@article{42ji2026normalized,
  title={Normalized solutions for the nonlinear biharmonic {C}hoquard Equation with {H}ardy--{L}ittlewood--{S}obolev Upper Critical Exponent},
  author={Ji, X. and Yang, Y.},
  journal={Taiwanese Journal of Mathematics},
  volume={30},
  number={1},
  pages={83--119},
  year={2026},
  publisher={Mathematical Society of the Republic of China}
}

@article{43rani2022multiple,
author = {Anu Rani and Sarika Goyal},
title={{Multiple solutions for biharmonic critical Choquard equation involving sign-changing weight functions}},
journal={Topological Methods in Nonlinear Analysis},
volume = {59},
number = {1},
pages = {221 -- 260},
year={2022}
}

@article{44Du2019UniquenessAN,
  title={Uniqueness and nondegeneracy of solutions for a critical nonlocal equation},
  author={Du, L. and Yang, M.},
  journal={Discrete and Continuous Dynamical Systems},
  year={2019},
  volume={39},
  number={10},
  pages={5847--5866},
  doi={10.3934/dcds.2019219}
}

@article{45article,
author = {Gao, F. and Shen, Z. and Yang, M.},
year = {2017},
month = {03},
pages = {26},
title = {On the critical Choquard equation with potential well},
volume = {38},
journal = {Discrete and Continuous Dynamical Systems - A},
doi = {10.3934/dcds.2018151}
}

@article{46battaglia2017existence,
  title={Existence of groundstates for a class of nonlinear {C}hoquard equations in the plane},
  author={Battaglia, L. and Van Schaftingen, J.},
  journal={Advanced Nonlinear Studies},
  volume={17},
  number={3},
  pages={581--594},
  year={2017},
  publisher={De Gruyter}
}

@article{47gao2020existence,
  title={Existence of solutions for critical {C}hoquard equations via the concentration-compactness method},
  author={Gao, F. and da Silva, E. D. and Yang, M. and Zhou, J.},
  journal={Proceedings of the Royal Society of Edinburgh Section A: Mathematics},
  volume={150},
  number={2},
  pages={921--954},
  year={2020},
  publisher={Royal Society of Edinburgh Scotland Foundation}
}

@article{48gao2020existence,
  title={Existence of multiple semiclassical solutions for a critical {C}hoquard equation with indefinite potential},
  author={Gao, F. and Yang, M. and Zhou, J.},
  journal={Nonlinear Analysis},
  volume={195},
  pages={111817},
  year={2020},
  publisher={Elsevier}
}

@article{49chen2026multiple,
  title={Multiple blowing-up solutions for the {C}hoquard type {B}r{\'e}zis-{N}irenberg problem in dimension three},
  author={Chen, W. and Wang, Z.},
  journal={Calculus of Variations and Partial Differential Equations},
  volume={65},
  number={3},
  pages={71},
  year={2026},
  publisher={Springer}
}

@article{51liu2024existence,
  title={Existence of solutions for nonlinear biharmonic {C}hoquard equations on weighted lattice graphs},
  author={Liu, Y. and Zhang, M.},
  journal={Journal of Mathematical Analysis and Applications},
  volume={534},
  number={2},
  pages={128079},
  year={2024},
  publisher={Elsevier}
}
\end{document}